\documentclass{amsart}
\usepackage{amssymb}

\usepackage[utf8]{inputenc}
\usepackage[T1]{fontenc}
\usepackage{mathpazo,courier}
\usepackage[scaled]{helvet}
\usepackage{bbm} 
\usepackage{lscape} 

\usepackage[usenames,dvipsnames]{xcolor}
\definecolor{darkblue}{rgb}{0.0, 0.0, 0.55}
\usepackage[pagebackref,colorlinks,linkcolor=BrickRed,citecolor=OliveGreen,urlcolor=darkblue,hypertexnames=true]{hyperref}

\usepackage{enumerate}

\newcommand\N{\mathbb N}

\newcommand\R{\mathbb R}
\newcommand\C{\mathbb C}

\renewcommand\P{\mathbb P}

\usepackage{ushort}
\newcommand\x{\ushort X}

\newcommand\al\alpha
\newcommand\be\beta
\newcommand\la\lambda
\newcommand\La\Lambda
\newcommand\De\Delta
\newcommand\Ph\Phi
\newcommand\Ps\Psi
\newcommand\ph\varphi
\newcommand\ps\psi
\newcommand\ze\zeta
\newcommand\rh\varrho
\newcommand\Om\Omega

\newcommand\ii{\mathbbm i}

\newcommand\dotcup{\mathbin{\dot\cup}}

\DeclareMathOperator\id{id}
\DeclareMathOperator\sgn{sgn}
\DeclareMathOperator\conv{conv}
\DeclareMathOperator\ev{ev}

\theoremstyle{definition}
\newtheorem{prop}{Proposition}[section]
\newtheorem{thm}[prop]{Theorem}
\newtheorem{cor}[prop]{Corollary}
\newtheorem{lem}[prop]{Lemma}
\newtheorem{df}[prop]{Definition}
\newtheorem{thmdf}[prop]{Theorem and Definition}

\newtheorem{rem}[prop]{Remark}

\newtheorem{ex}[prop]{Example}

\title[Optimization approaches to quadrature]{Optimization approaches to quadrature:\\
new characterizations of Gaussian quadrature on the line and
quadrature with few nodes on plane algebraic curves, on the plane and in higher dimensions}

\author[C. Riener]{Cordian Riener}
\address{Department of Mathematics and Statistics, Faculty of Science and Technology, University of Tromsø, 9037 Tromsø, Norway}
\address{Fachbereich Mathematik und Statistik, Universität Konstanz, 78457 Konstanz, Germany}

\email{cordian.riener@uit.no}

\author[M. Schweighofer]{Markus Schweighofer}
\address{Fachbereich Mathematik und Statistik, Universität Konstanz, 78457 Konstanz, Germany}
\email{markus.schweighofer@uni-konstanz.de}

\date{August 2, 2017}

\begin{document}
\begin{abstract}
Let $d$ and $k$ be positive integers. Let $\mu$ be a positive Borel measure on $\R^2$ possessing finite moments up to degree $2d-1$. If the support of $\mu$ is contained in an algebraic curve of degree $k$, then we show that there exists a quadrature rule for $\mu$ with at most $dk$ many nodes all placed on the curve (and positive weights) that is exact on all polynomials of degree at most $2d-1$. This generalizes both Gauss and (the odd degree case of) Szegő quadrature where the curve is a line and a circle, respectively, to arbitrary plane algebraic curves. We use this result to show that, without any hypothesis on the support of $\mu$, there is always a cubature rule for $\mu$ with at most $\frac32d(d-1)+1$ many nodes. In both results, we show that the quadrature or cubature rule can be chosen such that its value on a certain positive definite form of degree $2d$ is minimized. We characterize the unique Gaussian quadrature rule on the line as the one that minimizes this value or several other values as for example the sum of the
nodes' distances to the origin. The tools we develop should prove useful for obtaining similar results in higher-dimensional cases although at the present stage we can present only partial results in that direction.
\end{abstract}
\maketitle
\tableofcontents


\section{Motivation}\label{motivation}

\noindent
Initially designating the computation of areas and volumes, the terms \emph{quadrature} and \emph{cubature} now often stand for the numerical computation of
one-dimensional and two-dimensional integrals, respectively.  As a generic term for integrals over arbitrary dimension, the term \emph{quadrature} seems to be
more often used than \emph{cubature}. We will formally use both terms synonymously but use the latter term only when
the support of $\mu$ is a subset of $\R^2$.

\bigskip\noindent
By a \emph{measure} on $\R^n$ we always understand a nonnegative (i.e., unsigned) Borel measure on $\R^n$. Its \emph{support}
is the smallest closed subset of $\R^n$ whose complement has measure zero. Suppose $\mu$ is a measure on $\R^n$ and $f$ is a measurable
real valued function whose domain contains the support of $\mu$ and whose integral with respect to $\mu$ exists and is finite.
The aim is to compute the integral of $f$ numerically, i.e., the computation should be fast and should yield a good approximation of the actual
integral. Ideally, one should be able to have an error estimate for the approximation
and black box access to $f$ should be enough (in particular no information on
potentially existing derivatives or primitives of $f$ is needed).

\bigskip\noindent
A classical way of achieving this are \emph{quadrature rules} \cite{co1}.
They consist of finitely many points in $\R^n$ called \emph{nodes} together with associated real numbers
called \emph{weights}. The hope is that the weighted sum of the function values at the nodes approximates well the integral of $f$ with respect to $\mu$
(in the actual computation one has to deal with floating approximations, of course).
It is not indispensable but highly desirable that all weights are positive since this reflects the monotonicity of the integral, increases numerical stability and usually
allows for tighter error estimates \cite[Conclusion 3.19]{hac}.

\bigskip\noindent
We will therefore always insist on the weights to be positive.

\bigskip\noindent
Fix a nonnegative integer $d$ such that $\mu$ possesses finite moments up to degree $d$, i.e., all polynomials of degree $\le d$ have a finite integral with respect to $\mu$.
Then a quadrature rule for $\mu$ is often designed to yield the \emph{exact value} for the integral of an arbitrary polynomial of degree at most $d$ with respect to $\mu$.
In this way, for any function $f$ that can be well-approximated by a polynomial of degree at most $d$, simultaneously
\begin{itemize}
\item on the support of $\mu$ or at least where ``most of the mass of $\mu$ lies''
(i.e., on a measurable subset of the support of $\mu$ whose complement has reasonably small measure) and
\item on the nodes of the quadrature rule,
\end{itemize}
one and the same quadrature rule will give a good approximation for the integral.
In practice, one often works with small degree of exactness $d$ by splitting the domain of integration into many parts and integrating over each part separately.
One thus often has a good polynomial approximation of low or moderate degree with a good error analysis, e.g., by Taylor or Bernstein approximation.
Neither the subdivision (which can be adaptive to the problem) nor the error analysis is addressed in this article.

\bigskip\noindent
Note that one usually \emph{would not want to, could not, need not and does not} compute a polynomial approximation of $f$. Thus the real aim of a quadrature rule is not to
integrate polynomials which is an easy task anyway, as soon as the relevant moments of the measure are known. However, quadrature rules should be designed to
handle this easy task in the best possible way. That is why this article like a big part of the literature about quadrature is about integrating \emph{polynomials} with
quadrature rules.

\bigskip\noindent
Of course, it is important to have a small number of nodes to speed up the computation, especially when calculating nested
multiple integrals with respect to the same measure.

\bigskip\noindent
In addition, nodes that are far away from the origin should be evitated. Note that, due to the assumption on the
finiteness of the moments up to some degree, usually most of the mass of the measure cannot lie {too} far from the origin and therefore it seems unlikely that the given
function $f$ could be approximated at the same time on ``most of the mass of $\mu$ lies'' and at ``nodes that are far out''. Indeed, at first sight,
it seems reasonable to require that all nodes lie in the support of $\mu$. We will however give up on this frequently made requirement for one bad and several good reasons:
\begin{itemize}
\item The bad reason is that our proofs will in general not be able to guarantee the nodes to be contained in the support.
\item There are examples where the minimal number of nodes can only be reached if one does not insist on all
nodes being contained in the support, e.g., when one searches a quadrature rule of degree $6$
for the Lebesgue measure on the closed unit disk \cite[Theorem 1.1]{efp}.
\item For many important measures (e.g., the Gaussian measure on $\R^n$) the requirement to lie in the support is anyway empty and
it seems more important that the nodes are not too ``far out'', at least those with significant weight.
Our proofs will tend to ensure this by minimizing certain objective functions penalizing distant points, at least those with significant weight.
\item If some of the nodes lie outside the support, they still might lie close to where the mass of the measure is concentrated and there is a priori no obstacle to
approximate $f$ very well by a low degree polynomial simultaneously on both the support of $\mu$ and on the nodes (provided $f$ is defined on the nodes). 
\end{itemize}

\bigskip\noindent
Our main aim will be to show that for a given degree $d$ and a given measure $\mu$ on $\R^n$ having finite moments up to degree $d$, that
there exist quadrature rules with
two requirements:
\begin{itemize}
\item a small number of nodes and
\item certain quantities measuring how far the nodes are from most of the mass of the measure are minimized
\end{itemize}
Perhaps surprisingly, the two requirements seem not to be competing.
In fact, in our approach it even turns out that the second is key to the first.


\section{Basic ideas and tools}

\noindent
In this article, we will always be concerned with \emph{odd} degree $d$ (and we will therefore often write $2d-1$ instead of $d$). This is because the even case
is considerably more intricate (even though it seems also important for the applications). Although we are interested in quadrature rules exactly
integrating polynomials of
degree up to $2d-1$, in our proofs, polynomials of even degree $2d$ will be of outmost importance. One of the reasons for this is that polynomials in $n$ variables
of even degree can be seen not only as functions on $\R^n$ but also as functions on the bigger $n$-dimensional real projective space $\R\P^n$. One of the novelties in our
approach is that certain \emph{generalized quadrature rules} that are exact up to degree $2d$ will play a big role in our proofs.
Unlike classical quadrature rules they are allowed to have nodes in the hyperplane at infinity $\R\P^n\setminus\R^n$.
Such rules would in general not make any sense for non-polynomial functions for which one can hardly make sense of evaluating them at infinity. This will not be a problem,
since our penalty functions will finally prevent the nodes to lie in the infinitely far hyperplane.

\bigskip\noindent
This article will be mainly concerned with the question of how many nodes are necessary for measures on $\R^2$.
To the best of our knowledge, this is the first work giving non-trivial upper bounds for each fixed odd degree $2d-1$ of exactness and each measure on $\R^2$
with finite moments up to this degree (see Corollary \ref{curvecor} below).

\bigskip\noindent
Our basic strategy is as follows: Since we want to integrate only polynomials of degree $2d-1$ or less with respect to $\mu$, we keep from the measure only
the information on the moments of degree at most $2d-1$ (encoded for example in a truncated Riesz functional, see Definition \ref{dfqr}(b) below)
and perhaps some information on its support (encoded in a set $S$ which could for example be an algebraic curve in $\R^2$ containing the support of $\mu$).
Depending on this data, we set up two different types of optimization problems
over some encodings of generalized quadrature rules.
The objective is each time to minimize a certain penalty {function} that forces us to pay for nodes that are far from the origin, at least for
the ones with a not too small weight.
For the first type of optimization problems $(P_{L,S,f})$ on Page \ref{POP}, the generalized quadrature rules are hidden in
linear functionals on the vector space of polynomials of degree at most $2d$ and the penalty function could typically by the sum of the pure
moments of degree $2d$. For the second type of optimization problems $(Q_{L,f,m})$ on Page \ref{QOP}, the quadrature rules are almost directly encoded
and the penalty function could be for example the maximum distance to the origin of a node.

\bigskip\noindent
The amazing fact will then be that a certain clever handling of these optimization problems (which might involve iterated solving in the case of the first type
$(P_{L,S,f})$) will not only make sure that the nodes tend to be where we want them
to be for the applications (not too far from most of the mass of the measure, in particular not at infinity) but also we will obtain nonnegative polynomials with few real
zeros such that all nodes are contained
in these zeros. These polynomials will come out from the subtle duality theory of conic programming in the case of $(P_{L,S,f})$ and from a
nonsmooth Lagrange multiplier technique in the case of $(Q_{L,f,m})$.

\bigskip\noindent
Finally, all our upper bounds for the number of nodes will come from upper bounds on the number of real zeros of certain polynomials. We will use the deep results
\ref{petro} and \ref{petroo} of Petrovsky and Oleinik from real algebraic geometry, the variant \ref{newbezout} of Bézout's theorem and our taylor-made Lemma
\ref{lemma2} on the number of real roots of univariate polynomials subject to certain sign conditions.


\section{Notation and outline}

\bigskip\noindent
Throughout the article, let $\N:=\{1,2,3,\dots\}$ and $\N_0$ denote the set of positive and nonnegative integers, respectively. Let
$\x:=(X_1,\dots,X_n)$ always denote a tuple of variables and write $\R[\x]:=\R[X_1,\dots,X_n]$ for the real algebra of polynomials
in these variables. Most of our results are concerned with the cases $n\in\{1,2\}$ in which we will often write $X$ instead of $X_1$ and
$Y$ instead of $X_2$. Unless otherwise stated, all vector spaces will be real. For any vector space $V$, we denote by $V^*$ its
dual vector space consisting of all linear forms on $V$.
For $d\in\N_0$ we write $\R[\x]_d$ for the vector space of polynomials of (total) degree at most $d$.

\begin{df}\label{dfqr}
\begin{enumerate}[(a)]
\item
Let $L\in\R[\x]_d^*$. A \emph{quadrature rule} for $L$ is a function $w\colon N\to\R_{>0}$ defined on a finite set
$N\subseteq\R^n$ such that
\[L(p)=\sum_{x\in N}w(x)p(x)\]
for all $p\in\R[\x]_d$. We call the elements of $N$ the \emph{nodes} of the quadrature rule.
\item
Let $\mu$ be a measure on $\R^n$ and $d\in\N_0$. Suppose that $\mu$ has finite moments up to degree $d$, in other words
the \emph{$d$-truncated Riesz functional}
\[L_{\mu,d}\colon\R[\x]_d\to\R,\ p\mapsto\int_{\R^n}p~d\mu\]
is well-defined. A \emph{quadrature rule} of degree $d$ for $\mu$ is a quadrature rule for $L_{\mu,d}$.
\end{enumerate}
\end{df}

\bigskip\noindent
As already explained, the overall hope is that we will automatically get a small number of nodes when we
minimize, over all quadrature rules, a certain cost function making nodes far from the origin with large weight
very expensive. With the exception of Section \ref{line}, the minimization will be done over the closure of the
cone of linear forms possessing a quadrature rules in the sense of Definition \ref{dfqr}(a) (see Section \ref{conic} and Proposition
\ref{momclo}). The closure has to be taken for technical reasons in order to
fit into the theoretical framework of conic optimization. Its elements are
described in Proposition \ref{exgqr} by generalized quadrature rules with ``nodes at infinity''. These nodes at infinity will, however, be
``optimized away'' by our cost function (see Proposition \ref{primsol}(b)).

\bigskip\noindent
In Section \ref{line}, we will pursue a different idea where we minimize really over the the data of the quadrature rules themselves
instead of the associated linear forms. This time the cost function (for which there are many choices, see the beginning of Section \ref{line})
will just punish nodes far away from the origin and not take into account
the weights since the latter are bounded anyway by the prescribed total mass of the measure. This approach currently seems to work
only for quadrature on the real line but gives new interesting insights on this classical problem.

\bigskip\noindent
The following theorem from \cite{bt} is fundamental for the existence of quadrature rules. For a nice proof see also
\cite[Theorem 5.8]{la1} and \cite[Theorem 5.9]{la2}. Special cases of this theorem have been proven earlier by
Tchakaloff \cite{tch} and Putinar \cite{put}.

\begin{thm}[Bayer and Teichmann]\label{bavaria}
Each measure $\mu$ on $\R^n$ with finite moments up to degree $r\in\N_0$ possesses a quadrature rule of degree $r$ with all nodes
contained in the support of $\mu$.
\end{thm}

\noindent
Bayer and Teichmann give \[\binom{n+r}n=\binom{n+r}r=\dim\R[\x]_r=\dim\R[\x]_r^*\] as an upper bound on the number of nodes that is needed for finding such a quadrature rule \cite[Theorem 2]{bt}.
This bound is, however, trivial since it follows from Carathéodory's lemma applied to the vector space $\R[\x]_d^*$.
Recall that this lemma says that a conic (i.e., nonnegative linear) combination of
vectors in an $k$-dimensional vector space can always be written as a conic combination of $k$ of these vectors.
We will refer to this bound as the Carathéodory bound.

\noindent
This article will be concerned with variants of Theorem \ref{bavaria} mainly for $n\in\{1,2\}$ and exclusively for
odd degree $r=2d-1$ ($d\in\N$) with
much better bounds on the number of nodes on the expense of no longer insisting on the nodes being contained in the support of $\mu$.

\bigskip\noindent
For $n=1$, i.e., quadrature on the line, the above Carathéodory bound is $2d$. In Section \ref{curve},
we will reprove the classical bound of $d$ which
is often referred to as Gauss quadrature (see Theorem and Definition \ref{eindeutig} below) since
Gauss already treated the case where $\mu$ is the uniform measure on a compact interval \cite{g}.
More precisely, we show that minimizing the $2d$-th moment of the finitely supported measure corresponding to a quadrature rule
leads to at most $d$ nodes. This follows more or less already from the theory in the book of Dette and Studden \cite[Theorem 2.2.3]{ds} 
and has recently been rediscovered by Ryu and Boyd \cite{rb}, see our discussion in Section \ref{previous}.
However, our proof in Section \ref{curve} uses less machinery.

\bigskip\noindent
In Section \ref{line}, we complete our analysis in the case $n=1$ of Gaussian quadrature formulas
by a completely different technique that minimizes certain penalty functions depending only on the nodes
but not on the weights of the quadrature rules and again leaves at most $d$ non-zero weights.
Our result thus establishes new properties of the Gaussian quadrature rule used extensively in
applications that goes in the direction of the above mentioned desirable property of having nodes not too far from the origin.
At the same time a flattened version of our method suggested by Remark \ref{flattened} provides a new proof for the existence of the Gaussian
quadrature which could be included in an undergraduate textbook as a nice application of the Lagrange multiplier method.

\bigskip\noindent
For $n=2$, i.e., cubature on the plane, the above Carathéodory bound is $2d^2+d$. In Section \ref{plane}, we will prove the new bound of
$\frac32d^2-\frac32d+1=\frac32d(d-1)+1$. This will rely on a theorem of Petrovsky
on the topology of real algebraic curves as well as
a generalization of the mentioned Gauss quadrature from the line to any plane algebraic curve which we will prove in Section \ref{curve}:
Indeed, we will prove that if the support of $\mu$ is contained in a plane algebraic curve of degree $k$ then $dk$ nodes are sufficient, see
Corollary \ref{curvecor}.
The case of a circle has been proven before in \cite{jnt}.
The proof of this vast common generalization of Gauss and (the odd degree case of) Szegő quadrature
uses a strong version of Bézout's theorem on the number of intersection points of two plane algebraic curves.


\section{Cones of nonnegative polynomials, their duals and moments}\label{cone}

\noindent
A finite dimensional vector space carries a unique vector space topology, i.e., a Hausdorff
topology making addition and scalar multiplication continuous \cite[§2, Section 3, Theorem 2]{bou}. It is induced by any norm or scalar product.
By a \emph{cone} $C$ in a vector space $V$ we always mean a convex cone, i.e., a subset containing the origin
that is closed under addition and under multiplication with nonnegative scalars: $0\in C$, $C+C\subseteq C$ and $\R_{\ge0}C\subseteq C$.
For a cone $C$ in vector space $V$, its \emph{dual cone}
$C^*:=\{L\in V^*\mid L(C)\subseteq\R_{\ge0}\}$ is a cone in $V^*$.

\bigskip\noindent
In this section, we study the cone of polynomials in $n$ variables of degree at most $2d$ nonnegative on a set $S\subseteq\R^n$
\[P_{2d}(S):=\{p\in\R[\x]_{2d}\mid\forall x\in S:p(x)\ge0\}\]
and its dual cone
\[P_{2d}(S)^*=\{L\in\R[\x]_{2d}^*\mid L(P_{2d}(S))\subseteq\R_{\ge0}\}.\]

\noindent
For $\al\in\N^n$, we denote $|\al|:=\al_1+\dots+\al_n$ and $\x^\al:=X_1^{\al_1}\dotsm X_n^{\al_n}$, the latter being called \emph{monomials}.
For a polynomial $p=\sum_\al a_\al\x^\al$ ($a_\al\in\R$), we denote by $\deg p$ its degree, i.e.,
$\deg p=\max\{|\al|\mid a_\al\ne0\}$ if $p\ne0$ and $\deg0=-\infty$.
Polynomials all of whose monomials have exactly
the same degree
$d\in\N_0$ are called $d$-forms. They form a finite-dimensional vector space which we denote by $\R[\x]_{=d}$, i.e.,
\[\R[\x]_{=d}:=\left\{\sum_{|\al|=d}a_\al\x^\al\mid a_\al\in\R\right\}\]
so that
\[\R[\x]_d=\R[\x]_{=0}\oplus\dots\oplus\R[\x]_{=d}.\]
For $p\in\R[\x]_d$, we denote by
$p_0\in\R[\x]_{=0},\dots,p_d\in\R[\x]_{=d}$ the unique polynomials satisfying $p=p_0+\ldots+p_d$, and we call $p_k$ the
\emph{$k$-th homogeneous part} of $p$.
If in addition $\deg p=d$ (i.e., $p_d\ne0$), then $p_d$ is called the \emph{leading form} of $p$.

\bigskip\noindent
It is natural to see elements of $\R[\x]_{=2d}$ as functions on the $(n-1)$-dimensional real projective space
\[\R\P^{n-1}=\{\{-y,y\}\mid y\in S^{n-1}\}\] which arises by identifying antipodal points of
the $(n-1)$-dimensional unit sphere \[S^{n-1}:=\{y\in\R^n\mid\|y\|=1\}\] and carries the corresponding quotient topology induced from the sphere.
For a point $x=\{-y,y\}\in\R\P^{n-1}$ and
$p\in\R[\x]_{=2d}$, we write therefore
\[p(x):=p(y)=p(-y).\]
Note that $p(-y)=p(y)$ for $y\in\R^n$ since, more generally, 
$p(\la y)=\la^{2d}p(y)$ for all $\la\in\R$ and $y\in\R^n$.

\bigskip\noindent
In order to study $P_{2d}(S)$ and $P_{2d}(S)^*$ for $S\subseteq\R^n$, it turns out to be useful to consider first its homogeneous analogs
\[P_{=2d}(S):=\{p\in\R[\x]_{=2d}\mid\forall x\in S:p(x)\ge0\}\]
and
\[P_{=2d}(S)^*=\{L\in\R[\x]_{=2d}^*\mid L(P_{=2d}(S))\subseteq\R_{\ge0}\}\]
for $S\subseteq\R\P^{n-1}$.

\begin{df}\label{dfhqr}
Let $L\in\R[\x]_{=2d}^*$. A \emph{quadrature rule} for $L$ is a function $w\colon N\to\R_{>0}$ defined on a finite set
$N\subseteq\R\P^{n-1}$ such that
\[L(p)=\sum_{x\in N}w(x)p(x)\]
for all $p\in\R[\x]_{=2d}$. We call the elements of $N$ again the \emph{nodes}.
\end{df}

\noindent
The following lemma is probably well-known at least in the case $S=\R\P^{n-1}$ \cite[Theorem 3.7]{rez} but we include a proof for the convenience of the reader.

\begin{lem}\label{exhqr}
Let $S\subseteq\R\P^{n-1}$.
Each element of $P_{=2d}(S)^*$ possesses a quadrature rule with all nodes contained in the closure $\overline S$ of $S$.
\end{lem}

\begin{proof}
Clearly,
\begin{align*}
C:=\{L\in\R[\x]_{=2d}^*\mid~&\text{$L$ possesses a quadrature rule}\\
&\text{with all nodes contained in $\overline S$}\}
\end{align*}
is a cone and for its (pre-)dual $C^*$ we get
\begin{align*}
C^*&=\{p\in\R[\x]_{=2d}\mid\forall L\in C:L(p)\ge0\}\\
&=\{p\in\R[\x]_{=2d}\mid\forall x\in\overline S:p(x)\ge0\}\\
&=\{p\in\R[\x]_{=2d}\mid\forall x\in S:p(x)\ge0\}=P_{=2d}(S).
\end{align*}
Therefore $P_{=2d}(S)^*$ is the double dual $C^{**}$ of $C$ which equals $C$ \cite[Corollary 3.2.2 and Page 66]{ren} provided $C$ is closed.
We are thus reduced to showing that
$C$ is closed. Just as in the discussion after Theorem \ref{bavaria}, the weights in a quadrature rule for a linear form on $\R[\x]_{=2d}$
can be modified by Carathéodory's lemma so that at most $k:=\dim\R[\x]_{=2d}$ of its nodes are used. In particular, we can represent such
a quadrature rule by a matrix with $n+1$ rows and $k$ columns where each non-zero column encodes a node in $\R\P^{n-1}$ (represented by a vector in $S^{n-1}$) together with its weight. Now let $(L_m)_{m\in\N}$ be a sequence in $C$ converging
in $\R[\x]_{=2d}^*$. We have to show that its limit $L$ is again contained in $C$. Choose for each $m\in\N$ a
matrix $Q_m\in\R^{(n+1)\times k}$ representing a quadrature rule for $L_m$ with all nodes contained in $\overline S$.
The sequence $(Q_m)_{m\in\N}$ is bounded since the weights
appearing in a quadrature rule for $L_m$ cannot exceed $L_m((X_1^2+\dots+X_n^2)^d)$. By the
Bolzano-Weierstrass theorem, $(Q_m)_{m\in\N}$ possesses a convergent subsequence. By passing to the corresponding subsequence
of $(L_m)_{m\in\N}$
we can assume that $(Q_m)_{m\in\N}$ converges. Its limit $Q\in\R^{(n+1)\times k}$ represents a quadrature rule for $L$ with all nodes
contained in $\overline S$. Hence $L\in C$ as desired.
\end{proof}

\noindent
In the following, we identify $\R^n$ and $\R\P^{n-1}$ with a subset of $\R\P^n$ via the injective maps
\begin{align*}
\R^n\hookrightarrow\R\P^n,&\ x\mapsto\{\widetilde{(1,x)},-\widetilde{(1,x)}\}\\
\R\P^{n-1}\hookrightarrow\R\P^n,&\ \{x,-x\}\mapsto\{(0,x),(0,-x)\}
\end{align*}
where $\tilde x:=\frac x{\|x\|}\in S^n$ for $x\in\R^{n+1}\setminus\{0\}$. Under this identification, $\R\P^n$ is the disjoint union of its open subset
$\R^n$ and its closed subset $\R\P^{n-1}$. Moreover, both $\R^n$ and $\R\P^{n-1}$ are topological subspaces of $\R\P^n$. We call $\R\P^{n-1}$ the hyperplane at infinity and its elements the points at infinity of $\R\P^n$.
The \emph{projective closure} of a set $S\subseteq\R^n$ is its closure in $\R\P^n$. We denote the intersection of the projective closure of
$S\subseteq\R^n$ with the hyperplane at infinity $\R\P^{n-1}$ by $S_\infty$ and call it the \emph{the infinite reach} of $S$.
For each polynomial $p=\sum_{|\al|\le d}a_\al\x^\al\in\R[\x]_d$\quad($a_\al\in\R$), we denote by
\[p_{[d]}:=X_0^dp\left(\frac{X_1}{X_0},\dots,\frac{X_n}{X_0}\right)=\sum_{|\al|\le d}a_\al X_0^{d-|\al|}\x^\al\in\R[X_0,\x]_{=d}\]
its \emph{$d$-homogenization}. If $S$ is a closed subset of $\R^n$, then its projective closure is obviously $S\cup S_\infty$.

\begin{rem}\label{skal}
Let $x\in\R^n$ and $p\in\R[\x]_{2d}$. If we consider $x$ as an element of $\R\P^n$ (i.e., we replace $x$ by $\{\widetilde{(1,x)},-\widetilde{(1,x)}\}$)
and $p$ as a function on $\R\P^n$ (i.e., we replace $p$ by $p_{[2d]}$), then the evaluation of $p$ at $x$ undergoes a scaling with a
certain positive factor:
\[
p_{[2d]}\left(\{\widetilde{(1,x)},-\widetilde{(1,x)}\}\right)=p_{[2d]}\left(\frac{(1,x)}{\|(1,x)\|}\right)=\frac{p_{[2d]}(1,x)}{\|(1,x)\|^{2d}}=\frac{p(x)}{(1+\|x\|^2)^{d}}.
\]
\end{rem}

\begin{df}\label{dfgqr}
Let $L\in\R[\x]_{2d}^*$. A \emph{generalized quadrature rule} for $L$ is a function $w\colon N\to\R_{>0}$
defined on a finite set $N\subseteq\R\P^{n}$ such that
\[L(p)=\sum_{x\in N}w(x)p_{[2d]}(x)\]
or equivalently by Remark \ref{skal}
\[L(p)=\sum_{x\in N\cap\R^n}\frac{w(x)}{(1+\|x\|^2)^d}p(x)+\sum_{x\in N\cap\R\P^{n-1}}w(x)p_{2d}(x).\]
for all $p\in\R[\x]_{2d}$.
We call the elements of $N$, $N\cap\R^n$ and $N\cap\R\P^{n-1}$ the \emph{nodes}, the \emph{regular nodes} and
the \emph{nodes at infinity}, respectively.
\end{df}

\noindent
As discussed in Section \ref{motivation}, from numerical analysis there is a tendency that nodes of quadrature rules should not be far
away from the origin.
From this viewpoint, nodes at infinity are highly undesired. However, we will need them as a theoretical concept in the proofs even though we will
take care that they will finally not appear.

\begin{prop}\label{exgqr}
Let $S\subseteq\R^n$. Then
\begin{align*}
P_{2d}(S)^*=\{L\in\R[\x]_{2d}^*\mid&\text{$L$ possesses a generalized quadrature rule}\\
&\text{with all nodes contained in the projective closure of $S$}\}.
\end{align*}
\end{prop}

\begin{proof}
The inclusion ``$\supseteq$'' is trivial. To show the other inclusion,
let $L\in P_{2d}(S)^*$. One easily checks that
\[L_0\colon\R[X_0,\x]_{=2d}\to\R,\ p\mapsto L(p(1,X_1,\dots,X_n))\]
lies in $P_{=2d}(S)^*$ and therefore by Lemma \ref{exhqr} possesses a quadrature rule \[w\colon N\to\R_{>0}\] with
$N$ contained in the projective closure of $S$.
Now \[L(p)=L_0(p_{[2d]})=\sum_{x\in N}w(x)p_{[2d]}(x)\] for all $p\in\R[\x]_d$ by Definition \ref{dfhqr}. According to Definition \ref{dfgqr} this
means that $w$ is a generalized quadrature rule for $L$.
\end{proof}

\begin{rem}\label{intdua}
The interior of a cone $C$ in a finite-dimensional vector space consists exactly of those $x\in C$ such that
$L(x)>0$ for all $L\in C^*\setminus\{0\}$, confer for example \cite[Exercise 2.31(d)]{bv}.
\end{rem}

\begin{prop}\label{innerer}
Let $S\subseteq\R^n$ be closed. Then the interior of $P_{2d}(S)$ consists exactly of those $f\in\R[\x]_{2d}$ such that $f_{2d}(x)>0$ for all
$x\in S_\infty$ and $f(x)>0$ for all $x\in S$.
\end{prop}

\begin{proof}
Using Proposition \ref{exgqr}, the statement is easily seen to be equivalent to the following: The interior of $P_{2d}(S)$ consists exactly of
those $f\in\R[\x]_{2d}$ satisfying $L(f)>0$ for all $L\in P_{2d}(S)^*\setminus\{0\}$.
This is clear from the general theory of cones by Remark \ref{intdua}.
\end{proof}

\noindent
Recall that a cone $C$ is called \emph{pointed} if it does not contain a one-dimensional subspace, or equivalently $C\cap -C=\{0\}$.

\begin{prop}\label{spitz}
Let $S\subseteq\R^n$.
$P_{2d}(S)$ has non-empty interior and  $P_{2d}(S)^*$ is pointed.
\end{prop}

\begin{proof}
Replacing $S$ by its closure in $\R^n$ obviously does not change $P_{2d}(S)$. Assume therefore that $S$ is closed.
The first statement follows immediately from Proposition \ref{innerer} and the second follows from the first \cite[Exercise 2.31(e)]{bv}.
\end{proof}

\noindent
Let $S\subseteq\R^n$.
Note that $P_{2d}(S)$ is not necessarily pointed as the subspace
\[V_{2d}(S):=P_{2d}(S)\cap-P_{2d}(S)=\{f\in\R[\x]_{2d}\mid\forall x\in S:f(x)=0\}\]
of polynomials vanishing on $S$ might be nontrivial if $S$ has no interior. Since
$V_{2d}(S)$ is a linear subspace of $\R[\x]_{2d}$, it is in particular a cone and as such its dual cone is obviously
the subspace
\[V_{2d}(S)^*=\{L\in\R[\x]_{2d}\mid\forall p\in V_{2d}(S):L(p)=0\}\]
of $\R[\x]_{2d}^*$.

\begin{prop}\label{hull}
Let $S\subseteq\R^n$. In $\R[\x]_{2d}^*$, the linear subspace generated by the cone $P_{2d}(S)^*$ equals
$V_{2d}(S)^*$.
\end{prop}

\begin{proof}
The following is a well-known exercise from the theory of cones:
If $C$ and $D$ are closed cones in a finite-dimensional vector space, then $(C\cap D)^*$ is the closure of $C^*+D^*$.
If we set $C:=P_{2d}(S)$ and $D:=-P_{2d}(S)$ then
\[C^*+D^*=P_{2d}(S)^*+(-P_{2d}(S))^*=P_{2d}(S)^*-P_{2d}(S)^*\]
is a subspace and therefore closed.
\end{proof}

\noindent
For $d\in\N_0$ and $S\subseteq\R^n$, we call the cone
\begin{align*}
M_{d}(S):=\{L_{\mu,d}\mid~&\text{$\mu$ {(nonnegative)} measure on $\R^n$ with support contained in $S$}\\
&\text{and finite moments up to degree $d$}\}
\end{align*}
the \emph{$d$-truncated $S$-moment cone}. Since integrals of nonnegative functions are nonnegative, we have
$M_{2d}(S)\subseteq P_{2d}(S)^*$.
The Bayer-Teichmann Theorem \ref{bavaria} says that
\[M_{d}(S)=\{L\in\R[\x]_d^*\mid\text{$L$ has a quadrature rule with all nodes contained in $S$}\}.\]
In the special case where $S$ is compact and $d$ is even this follows also readily from Proposition \ref{exgqr}.

\begin{prop}\label{momclo}
Let $S\subseteq\R^n$.
The closure of $M_{2d}(S)$ is $P_{2d}(S)^*$.
\end{prop}

\begin{proof}
Let $L\in P_{2d}(S)^*$ and $U$ be a neighborhood of $L$ in $\R[\x]_{2d}^*$.
By Proposition \ref{exgqr} $L$ possesses a generalized quadrature rule with all nodes contained in the projective closure of $S$.
By continuity, we find an element of $U$ possessing a quadrature rule with all nodes in $S$. Indeed, it suffices to
replace all nodes in the quadrature
rule for $L$ by close-by nodes lying in $S$ and to leave the weights unchanged.
\end{proof}

\noindent
The inclusion $M_{2d}(S)\subseteq P_{2d}(S)^*$ for $d\in\N_0$ and $S\subseteq\R^n$ is in general strict as the next example
shows. In particular, $M_{2d}(S)$ is in general not closed whereas $P_{2d}(S)^*$ is. This also shows that quadrature rules for elements of
$P_{2d}(S)^*$ require in general nodes at infinity in accordance with Proposition \ref{exgqr}.

\begin{ex}
The linear form
\[L_\infty\colon\R[X]_2\to\R, \ aX^2+bX+c\mapsto a\qquad(a,b,c\in\R)\]
lies in $P_2^*(\R)$ but clearly not in $M_2(\R)$ since a measure with zero mass is the zero measure.
By Proposition \ref{momclo}, $L_\infty$ is a limit of a sequence of $2$-truncated Riesz functionals of measures on $\R$. Here
it is very easy to see that these measures can be taken to be normal distributions.
\end{ex}


\section{The conic optimization problem}\label{conic}

\noindent
Throughout the section, fix
\begin{itemize}
\item $n,d\in\N$,
\item a closed set $S\subseteq\R^n$,
\item $L\in M_{2d-1}(S)$ and
\item $f\in\R[\x]_{2d}$.
\end{itemize}
Typically, $S$ would be the support of a measure $\mu$ on $\R^n$ with finite moments up to degree $2d-1$,
$L:=L_{\mu,2d-1}$ its truncated Riesz-functional
and $f$ a quickly growing polynomial like $X_1^{2d}+\dots+X_n^{2d}$. The overall
idea is to find a quadrature rule for $L$ with few nodes by extending $L$ to some $\La\in M_{2d}(S)$ with minimal $\La(f)$ and hoping
that \emph{each} quadrature rule for $\La$ has few nodes. In order to fit into the framework of conic programming, we have to work
\emph{with the closure} of $M_{2d}(S)$ which is $P_{2d}(S)^*$ by Proposition \ref{momclo}.

\bigskip\noindent
Now consider the following primal-dual pair of conic optimization problems
\cite[Subsection 3.1]{ren} {(the optimization variables are $\La$ in the primal and $q$ in the dual and the appropriate choice
of $f$ will be discussed in the sequel)}:

\[\label{POP}
\begin{array}{lllll}
(P_{L,f,S})&\text{minimize}&\La(f)\\
&\text{subject to}&\La\in P_{2d}(S)^*\\
&&\La|_{\R[\x]_{2d-1}}=L\\
\\
(D_{L,f,S})&\text{maximize}&L(q)\\
&\text{subject to}&q\in\R[\x]_{2d-1}\\
&&f-q\in P_{2d}(S)
\end{array}
\]
Note that our optimization problems have been formulated without reference to a scalar product \cite[Page 66]{ren} but one could write them
simply in terms of matrices and vectors by choosing for example the monomial basis in $\R[\x]_{2d-1}$ and its dual basis in $\R[\x]_{2d-1}^*$.

\bigskip\noindent
Although, this follows from the general theory of conic programming \cite[Page 66]{ren},
we give a direct argument for weak duality of $(P_{L,f,S})$ and $(D_{L,f,S})$:
Denote by $P_{L,f,S}^*$ and $D_{L,f,S}^*$ the optimal values of these optimization problems
which are defined as an infimum and a supremum in the ordered set
$\{-\infty\}\cup\R\cup\{\infty\}$, respectively \cite[Page 65]{ren}. Indeed, we have \emph{weak duality} $P_{L,f,S}^*\ge D_{L,f,S}^*$
for if $\La$ is feasible for $(P_{L,f,S})$ and $q$ for $(D_{L,f,S})$, then $\La(f)\ge L(q)$ since
\[\La(f)-L(q)=\La(f)-\La(q)=\La(f-q)\subseteq \La(P_{2d}(S))\subseteq\R_{\ge0}.\]

\begin{prop} \label{fies}
$(P_{L,f,S})$ is feasible.
\end{prop}

\begin{proof}
By definition of $M_{2d-1}(S)$, we have $L=L_{\mu,2d-1}$ for a measure $\mu$ on $\R^n$ with support contained in $S$ and finite moments up to degree $2d-1$. By the Bayer-Teichmann Theorem \ref{bavaria}, we may assume that $\mu$ has finite support and therefore also finite moments of degree $2d$. Now $\La:=L_{\mu,2d}\in M_{2d}(S)\subseteq P_{2d}(S)^*$ is feasible for $(P_{L,f,S})$.
\end{proof}

\begin{lem}\label{rein}
If $f_{2d}(x)>0$ for all $x\in S_\infty$, then there is 
$R\in\N$ such that $f(x)>0$ for all $x\in S$ with $\|x\|\ge R$.
\end{lem}

\begin{proof}
Denote by $S':=S\cup S_\infty\subseteq\R\P^n$ the projective closure of $S$. We prove the contraposition: Suppose for all
$R\in\N$ there
is $x_{R}\in S$ such that $\|x_{R}\|\ge R$ and $f(x_{R})\le0$. We show that there is $x\in S_\infty$ with $f_{2d}(x)\le0$.
Viewed as a sequence in the compact projective space $\R\P^n$, $(x_{R})_{R\in\N}$ has a convergent subsequence and thus can be assumed without loss of generality to converge. Its limit $x$ lies in $S_\infty$. 
Since the $2d$-homogenization $f_{[2d]}\in\R[X_0,\x]$ of $f$ is nonpositive on every $x_{R}$ (see Remark \ref{skal}),
it is nonpositive also on $x$. But then $0\ge f_{[2d]}(x)=f_{2d}(x)$ where we consider $x$ first as an element of $\R\P^n$ and then as an
element of its hyperplane at infinity $\R\P^{n-1}$.
\end{proof}

\noindent
Following the usual terminology, we call a feasible solution of $(P_{L,f,S})$ or  $(D_{L,f,S})$ \emph{strictly} feasible if there is a feasible
solution such that the respective cone membership constraint is fulfilled even with the cone replaced by its interior.
This is not to confuse with notion of
\emph{strong feasibility} \cite[Page 73]{ren} (confer also the proof of \cite[Theorem 3.2.6]{ren}):
\begin{itemize}
\item $(P_{L,f,S})$ is called \emph{strongly} feasible if
there is a neighbourhood $U$ of $L$ in $\R[\x]_{2d-1}^*$ such that $(P_{\tilde L,f,S})$ is feasible for all $\tilde L\in U$ 
\item $(D_{L,f,S})$ is called \emph{strongly} feasible if
there is a neighbourhood $U$ of $f$ in $\R[\x]_{2d}$ such that $(D_{L,\tilde f,S})$ is feasible for all $\tilde f\in U$.
\end{itemize}
For $(D_{L,f,S})$, it is trivial that strict feasibility implies strong feasibility. For the primal problem $(P_{L,f,\R^n})$, the same holds true since
the linear operator \[\R[\x]_{2d}^*\to\R[\x]_{2d-1}^*,\ \La\mapsto\La|_{\R[\x]_{2d-1}}\] appearing in its specification is surjective.

\begin{prop}\label{dualstrict}
Suppose that $f_{2d}(x)>0$ for all $x\in S_\infty$. Then $(D_{L,f,S})$ is strictly feasible.
\end{prop}

\begin{proof}
Choose $N\in\N$ such that $f(x)>0$ for all $x\in S$ with $\|x\|\ge N$ according to Lemma \ref{rein}. By compactness of
$\{x\in S\mid\|x\|\le N\}$,
we can choose $q\in\R_{\le0}\subseteq\R[\x]_{2d-1}$ such that $f-q>0$ for all $x\in S$. Now $f-q$ lies in the interior of $P_{2d}(S)$ by
Lemma \ref{innerer}.
\end{proof}

\noindent
The proof of Part (b) of the next proposition
shows that ``nodes at infinity are optimized away'' in $(P_{L,f,S})$ under mild conditions as advertised in the introduction.

\begin{prop}\label{primsol}
Let $f_{2d}(x)>0$ for all $x\in S_\infty$.
\begin{enumerate}[(a)]
\item Then $(P_{L,f,S})$ possesses an optimal solution.
\item Each such lies in $M_{2d}(S)$, i.e., possesses a quadrature rule with all nodes contained in $S$.
\end{enumerate}
\end{prop}

\begin{proof}
(a) It suffices by \cite[Theorem 3.2.8]{ren} to show that $(P_{L,f,S})$ is feasible and
$(D_{L,f,S})$ is \emph{strongly} feasible. This follows from Proposition \ref{fies} and Proposition \ref{dualstrict}.

\smallskip
(b) Let $\La$ be an optimal solution of $(P_{L,f,S})$. By Proposition \ref{exgqr}, $\La$ possesses a generalized quadrature rule
$w\colon N\to\R_{>0}$ in the sense of Definition \ref{dfgqr} with $N\subseteq S\cup S_\infty$.
Now simply remove all nodes at infinity, i.e., look at the generalized
quadrature rule $w|_{N\cap S}$ (which corresponds up to the scaling issue addressed in Remark \ref{skal} to a usual quadrature rule
in the sense of Definition \ref{dfhqr}). Then it is easy to see that the unique linear form $\La'\colon\R[\x]_{2d}\to\R$ having the
generalized quadrature
rule $w|_{N\cap S}$ agrees with $\La$ on $\R[\x]_{2d-1}$ and therefore is also feasible for $(P_{L,f,S})$.
Comparing the values of the objective function on $\La$ and $\La'$, yields
\[\La(f)=\La'(f)+\sum_{x\in N\cap S_\infty}w(x)f_{2d}(x).\]
By the optimality of $\La$, it follows that $N\cap S_\infty=\emptyset$ since $w(x)f_{2d}(x)>0$ for all $x\in S_\infty$.
\end{proof}

\begin{df}
For any $h\in\R[\x]$, we denote by $Z(h):=\{x\in\R^n\mid h(x)=0\}$ its \emph{real zero set}.
\end{df}

\begin{rem}\label{strategy}
The rough basic idea behind our approach is the following: We would like to find a quadrature rule for $L$ with few nodes.
It is easy to see that for any $h\in P_{2d-1}(S)$ with $L(h)=0$, every node of every quadrature rule for $L$ must lie
in $Z(h)\cap S$. If there is such an $h$ with small $Z(h)\cap S$, then we are done. In general, we can, however not expect
that such an $h$ exists. So our strategy is to extend $L$ to $\La\in M_{2d}(S)$ for which there now is an $h\in P_{2d}(S)$ with
$\La(h)=0$
such that $Z(h)\cap S$ has few elements. Then every node of every quadrature rule for $\La$ must lie in $Z(h)\cap S$. Since every
quadrature rule for $\La$ is also a quadrature rule for $L$, we are done.
\end{rem}

\begin{lem}\label{nichtgross}
If $\#S>\dim\R[\x]_{2d-1}=\binom{n+2d-1}n$, then $\R[\x]_{2d-1}+V_{2d}(S)$ is not all of $\R[X]_{2d}$.
\end{lem}

\begin{proof}
Suppose $\R[\x]_{2d-1}+V_{2d}(S)=\R[X]_{2d}$. We show $\#S\le\dim\R[\x]_{2d-1}$.
Let $I$ be the ideal generated in $\R[\x]$ by $V_{2d}(S)$. By hypothesis each monomial of degree $2d$ is modulo the
ideal $I$ congruent to a polynomial of degree at most $2d-1$. By induction, the same applies to each monomial regardless of its
degree. Therefore
the residue classes of monomials of degree at most $2d-1$ generate $\R[\x]/I$ as a vector space. By elementary algebraic geometry,
this implies that the variety $\{x\in\C^n\mid\forall p\in I:p(x)=0\}$ has at most
$\dim\R[\x]_{2d-1}$ many points (see for example \cite[Corollary 3.8 and Proposition 3.11(f)]{kun}).
But $S$ is contained in this variety.
\end{proof}

\begin{lem}\label{delta}
If $\#S>\dim\R[\x]_{2d-1}=\binom{n+2d-1}n$, then
there is $\De\in\R[\x]_{2d}^*\setminus\{0\}$ such that $\De|_{\R[\x]_{2d-1}}=0$ and $\De(V_{2d}(S))=\{0\}$.
\end{lem}

\begin{proof} This follows directly from Lemma \ref{nichtgross} with linear algebra.
\end{proof}

\noindent
The next theorem will be particularly useful for finding quadrature rules with few nodes
for measures supported on plane algebraic curves, see Section \ref{curve}. Its strong part is that Part (a) works under a very mild
hypothesis, for example if $S$ is infinite. The price to pay for this is, however, that $Z(h)\cap S$ might in general not be a small enough
set according to our strategy explained in Remark \ref{strategy} (in particular, $Z(h)\cap S$ might be infinite).

\begin{thm}\label{qt}
\begin{enumerate}[(a)]
\item If $\#S>\dim\R[\x]_{2d-1}=\binom{n+2d-1}n$ and $f_{2d}(x)>0$ for all $x\in S_\infty$, then
$(P_{L,f,S})$ possesses an optimal solution $\La$ on the relative boundary of $P_{2d}(S)^*$, i.e.,
the boundary of $P_{2d}(S)^*$ in $P_{2d}(S)^*-P_{2d}(S)^*\overset{\ref{hull}}=V_{2d}(S)^*$.
\item For each feasible solution $\La$ of $(P_{L,f,S})$ on the relative boundary of $P_{2d}(S)^*$,
there is $h\in P_{2d}(S)\setminus V_{2d}(S)$ (i.e. $Z(h)\cap S$ is a proper subset of $S$, cf. Remark \ref{strategy}) with $\La(h)=0$.
\end{enumerate}
\end{thm}

\begin{proof}
(a) Suppose $\#S>\dim\R[\x]_{2d-1}$ and $f_{2d}(x)>0$ for all $x\in S_\infty$.
By Proposition \ref{primsol}(a) we can choose an optimal solution $\La$ of $(P_{L,f,S})$. If it lies on the relative boundary of $P_{2d}(S)^*$,
we are done. Therefore suppose now that $\La$ lies in the relative interior of $P_{2d}(S)^*$.
Choose $\De$ like in Lemma \ref{delta}. Then we claim that we find $\la\in\R\setminus\{0\}$ such that $\La+\la\De$ is also
an optimal solution of
$(P_{L,f,S})$. Indeed, for $\la\in\R$ close to $0$, we see that $\La+\la\De$ is feasible for $(P_{L,f,S})$ since $\De$ lies in $V_{2d}(S)^*$.
Because of the optimality of $\La$, we have that $(\La+\la\De)(f)\ge \La(f)$ for $\la$ close to $0$. Hence
$\De(f)=0$. So for each $\la\in\R$, $\La+\la\De$ is again an optimal solution for $(P_{L,f,S})$ provided it is feasible. It suffices to show that
there is $\la\in\R$ such that $\La+\la\De$ lies on the relative boundary of $P_{2d}(S)^*$. Assume this to be false. Then the whole line
$\{\La+\la\De\mid\la\in\R\}$ would lie in $P_{2d}(S)^*$. Because $P_{2d}(S)^*$ is closed, one sees easily that this would imply that the line
$\R\De$ lies also in $P_{2d}(S)^*$. But this contradicts Proposition \ref{spitz}.

\smallskip
(b) Let $\La$ be a feasible solution of $(P_{L,f,S})$ on the relative boundary of $P_{2d}(S)^*$.
By the supporting hyperplane theorem (in its conic version which follows easily from the affine version given in
\cite[Page 51]{bv}) there is a non-zero linear form $G\colon V_{2d}(S)^*\to\R$ such that $G(P_{2d}(S)^*)\subseteq\R_{\ge0}$ and
$G(\La)=0$. We can extend $G$ to a linear form $H\colon\R[\x]_{2d}^*\to\R$. Then $H$ lies in the double dual $R[\x]_{2d}^{**}$ of the vector
space $\R[\x]_{2d}$ and actually in the double dual $P_{2d}(S)^{**}$ of the cone $P_{2d}$ but not in $V_{2d}(S)^{**}$ since
$H|_{V_{2d}(S)^*}=G\ne0$. Using the canonical isomorphism $\Psi\colon\R[\x]_{2d}\to\R[\x]_{2d}^{**},\ p\mapsto(\La\mapsto\La(p))$,
we get from the closedness of $P_{2d}(S)$ by \cite[Exercise 2.31(f)]{bv} for $h:=\Psi^{-1}(H)$ that
$h\in P_{2d}(S)\setminus V_{2d}(S)$ and $\La(h)=0$.
\end{proof}

\begin{lem}\label{primalstrict}
Suppose $L(p)>0$ for all $p\in P_{2d-2}(\R^n)\setminus\{0\}$. Then $(P_{L,f,\R^n})$ is strictly feasible.
\end{lem}

\begin{proof}
By Lemma \ref{fies}, there exists a feasible solution $\La_0$ of $(P_{L,f,\R^n})$.
Since the cone $P_{=2d}(\R\P^{n-1})\subseteq\R[\x]_{=2d}$ is pointed,
there is moreover $\La_\infty\in\R[\x]_{=2d}^*$ such that $\La_\infty(p)>0$ for all $p\in P_{=2d}(\R\P^{n-1})\setminus\{0\}$. We claim that
\[\La\colon\R[\x]_{2d}\to\R,\ p\mapsto\La_\infty(p_{2d})+\La_0(p)\]
is a strictly feasible solution of $(P_{L,f,\R^n})$.
It is clear that \[\La|_{\R[\x]_{2d-1}}=\La_0|_{\R[\x]_{2d-1}}=L.\]
By the (dual version of) Remark \ref{intdua}, it remains to check that
 $\La(p)>0$ for all $p\in P_{2d}(\R^n)\setminus\{0\}$.
So let $p\in P_{2d}(\R^n)\setminus\{0\}$. It is easy to see that $p$ must be of even degree $2k$ for some $k\in\{0,\dots,d\}$. If
$k\le d-1$ then $p\in P_{2d-2}(\R^n)\setminus\{0\}$ and therefore $\La(p)=\La_\infty(p_{2d})+\La_0(p)=0+\La_0(p)>0$ by hypothesis.
Therefore we can assume $k=d$. But then it is easy to see that $p_{2d}\in P_{=2d}(\R\P^{n-1})\setminus\{0\}$ and therefore
$\La(p)=\La_\infty(p_{2d})+\La_0(p)\ge\La_\infty(p_{2d})>0$.
\end{proof}

\noindent
The next theorem will serve in Section \ref{plane} to find cubature rules with few nodes
for measures supported on the plane. Part (a) needs a strong hypothesis which is fulfilled, however,
for example if $L=L_{\mu,2d-1}$ for some non-zero
measure $\mu$ on $\R^n$ that has a density with respect to the Lebesgue measure on $\R^n$.
The reward is that in Part (b) the set $Z(h)$ has a good chance to be small like required by the strategy explained in
Remark \ref{strategy}. At least this will be true if one chooses $f$ such that $f_{2d}=X_1^{2d}+\dots+X_n^{2d}$ as we will see in Proposition \ref{norm2dbound} below. It is interesting to compare the result to Theorem \ref{qt}.

\begin{thm}\label{ct}
\begin{enumerate}[(a)]
\item Suppose $L(p)>0$ for all $p\in P_{2d-2}(\R^n)\setminus\{0\}$ and $f_{2d}(x)>0$ for all $x\in\R\P^{n-1}$.
Then both $(P_{L,f,\R^n})$ and $(D_{L,f,\R^n})$ possess optimal solutions with the same value.
\item If both $(P_{L,f,\R^n})$ and $(D_{L,f,\R^n})$ possess an optimal solution with the same value, then for each optimal solution
$\La$ of $(P_{L,f,\R^n})$ there is $h\in P_{2d}(\R^n)$ with $\La(h)=0$ and $h_{2d}=f_{2d}$.
\end{enumerate}
\end{thm}

\begin{proof}
(a) Both $(P_{L,f,\R^n})$ and $(D_{L,f,\R^n})$ are strictly feasible by Lemma \ref{primalstrict} and Proposition \ref{dualstrict}, respectively.
From the strict feasibility of  $(P_{L,f,\R^n})$ and the surjectivity of the
linear operator $\R[\x]_{2d}^*\to\R[\x]_{2d-1}^*,\ \La\mapsto\La|_{\R[\x]_{2d-1}}$ appearing in it,
the optimal values of the primal and the dual problem therefore agree, i.e., $P_{L,f,\R^n}^*=D_{L,f,\R^n}^*$
\cite[Corollary 3.2.7]{ren}. By \cite[Theorem 3.2.8]{ren}, it suffices to show that both $(P_{L,f,\R^n})$ and $(D_{L,f,\R^n})$ are strongly
feasible. But they are even strictly feasible as already mentioned.

\smallskip
(b) Let $\La$ be an optimal solution of $(P_{L,f,\R^n})$ and $q$ an optimal solution of $(D_{L,f,\R^n})$ such that $\La(f)=L(q)$.
Set $h:=f-q$. Then $h\in P_{2d}(S)$, $\La(h)=\La(f)-\La(q)=\La(f)-L(q)=0$ and $h_{2d}=f_{2d}$.
\end{proof}

\noindent
The following example gives a first impression of how to apply the results from this section. We reprove the existence of Gaussian
quadrature for measures on the real line (see Theorem and Definition \ref{eindeutig} below) whose existence was already proved by Gauss
if $\mu$ is the uniform distribution on a compact interval \cite{g}. We will return to the topic of Gaussian quadrature on the line
in Section \ref{line}.

\begin{ex}\label{uniex}
We apply Theorems \ref{qt} and \ref{ct} to the case $n:=1$ and $S:=\R$.
The obvious choice for $f$ is $f:=X^{2d}$ and this is essentially equivalent to all other reasonable choices of $f$.
Now Theorem \ref{qt} yields an optimal solution $\La$ of $(P_{L,X^{2d},\R})$ and a non-zero polynomial $h\in P_{2d}(\R)$ such that
$\La(h)=0$. In this case, Theorem~\ref{ct} yields the same: If \ref{ct}(a) is not applicable, then there is such an $h$ of degree at most $2d-2$,
or it is applicable and then \ref{ct}(b) yields a such an $h$ with {$2d$-th homogeneous part}
$h_{2d}=X^{2d}$. Now $h$ can have at most $d$
different roots since each of them is double by Remark~\ref{nonn} below. By Proposition \ref{primsol}(b), $\La$ possesses a
quadrature rule all of whose nodes are roots of $h$ (cf. Remark \ref{strategy}). So we have found a quadrature rule for $L$
(realizing $2d$ moments of degrees $0\dots,2d-1$ of a given measure on the real line) with at most $d$ nodes (and of course
as many positive weights). Among all quadrature rules for $L$, this quadrature rule minimizes the moment of degree $2d$ which should be considered a good thing according to the discussion in the introduction.
\end{ex}


\section{General results}\label{general}

\noindent
In this section, we summarize the immediate
implications of Theorem \ref{ct} for quadrature rules of measures on $\R^n$. The results of this section
will mostly be very technical but the hope is that they could serve as a basis for subsequent work. At least for $n=2$ this is demonstrated
in the following two sections. For $n\ge3$ more algebraic geometry will probably have to enter the picture to make this section more fruitful.

\begin{df}\label{doubleroot}
Denote by $\C[\x]$ the complex algebra of complex polynomials in $n$ variables $X_1,\dots,X_n$.
We call $x\in\C^n$ a \emph{double root} of a polynomial $p\in\C[\x]$ if
\[p(x)=\frac{\partial p}{\partial X_1}(x)=\ldots=\frac{\partial p}{\partial X_n}(x)=0.\]
\end{df}

\begin{rem}\label{nonn}
Suppose $h\in\R[\x]$ with $h(x)\ge0$ for all $x\in\R^n$. Then each element of $Z(h)$ is a double root of $h$.
\end{rem}

\noindent
Note that the next proposition would be completely wrong if one replaced the condition $h_{2d}=X_1^{2d}+\ldots+X_n^{2d}$ by, e.g.,
$h_{2d}=(X_1^2+\ldots+X_n^2)^d$. Indeed, if we take for example $h:=(X_1^2+\ldots+X_n^{2d}-1)^2$, then $Z(h)$ is the unit sphere which is
infinite if $n\ge2$.

\begin{prop}\label{norm2dbound}
Suppose $h\in P_{2d}(\R)$ such that $h_{2d}=X_1^{2d}+\dots+X_n^{2d}$. Then \[\#Z(h)\le(2d-1)^n.\]
\end{prop}

\begin{proof}
Consider the ideal
\[I:=\left(\frac{\partial f}{\partial X_1},\dots,\frac{\partial f}{\partial X_n}\right)\subseteq\C[\x],\]
and its affine algebraic variety
\[V:=\left\{x\in\C^n\mid\frac{\partial f}{\partial X_1}(x)=\dots=\frac{\partial f}{\partial X_n}(x)=0\right\}.\]
We proceed in several steps:
\begin{enumerate}[(a)]
\item $Z(h)\subseteq V$
\item The quotient algebra $\C[\x]/I$ has a complex vector space dimension of at most $(2d-1)^n$.
\item $\#V<\infty$
\item$ \#V\le(2d-1)^n$
\end{enumerate}
Claim (a) is clear from Remark \ref{nonn} and Definition \ref{doubleroot}.
Using that the leading form of $\frac{\partial f}{\partial X_i}$ is $X_i^{2d-1}$,
one shows easily that each monomial $\x^\al$ with $\al\in\N_0^n$ is modulo $I$ congruent to another monomial
$\x^\be$ with $\be\in\{0,\dots,2d-2\}$. Now (b) follows since the residue classes of these $\x^\be$ generate
$\C[\x]/I$ as a vector space. From (b) one deduces immediately that $I$ contains non-zero univariate polynomials in each variable.
This implies (c) since this leaves only finitely many possible values for each component of an element of $V$.
To prove finally (d), consider the map that sends a polynomial $\C[\x]$ to the function it induces on $V$. This map is an algebra
homomorphism. It is surjective by polynomial interpolation since $V$ is finite by (c). Moreover, it contains $I$ in its kernel and therefore
induces a surjective algebra homomorphism $\C[\x]/I\to\C^V$. In particular, this is a surjective $\C$-linear map. Hence
\[\#V=\dim_\C(\C^V)\le\dim_\C(\C[\x]/I)\overset{(b)}\le(2d-1)^n.\]
\end{proof}

\noindent
Note that in the above proof one could replace $\C$ by $\R$ everywhere. However, we chose to work with $\C$ in order to suggest
one of the possible reasons why we should not expect the bound to be optimal. In the proof, we could have used much more machinery instead
of our elementary arguments: To prove Claim (c), we could have used Gröbner basis theory \cite[Chapter 5, \S 3, Theorem 6]{clo}.
Then we could have deduced (d) from (c) by the higher-dimensional projective \cite[Chapter 4, Section 2.1]{sha} or affine
\cite{sch} analogue of Bézout's theorem. It is remarkable and fortunate
that we do not need this in our case. In principle, it is conceivable that in the future
one might find good reasons to replace $X_1^{2d}+\ldots+X_n^{2d}$ by another $2d$-form in the statement of Proposition \ref{norm2dbound}
doing the job. In that case, these nontrivial ingredients from algebraic geometry might come into play. At the moment, we can, however,
hardly imagine any better choice for the $2d$-form.

\begin{df}\label{rhodef}
For $n,d\in\N$, we set
\[\rh(n,2d):=\max\{\#Z(h)\mid h\in P_{2d}(\R^n),h_{2d}=X_1^{2d}+\dots+X_n^{2d}\}\overset{\ref{norm2dbound}}\le(2d-1)^n.
\]
\end{df}

\begin{rem}\label{roughlowerbound}
For $n,d\in\N$, we have \[d^n\le\rh(n,2d)\le(2d-1)^n\] since $Z(h)=\{1,\dots,d\}^n$ for
\[h:=(X_1-1)^2\dotsm(X_1-d)^2+\ldots+(X_n-1)^2\dotsm(X_n-d)^2\in P_{2d}(\R^n).\]
\end{rem}

\begin{ex}\label{rho1}
Clearly $\rh(1,2d)=d$ for $d\in\N$ since a polynomial $h\in P_{2d}(\R)\setminus\{0\}$ has at most $d$ real roots since
each such root is double by Remark \ref{nonn}.
\end{ex}

\begin{thm}\label{genthm}
Let $d\in\N$ and $L\in M_{2d-1}(\R^n)$ such that $L(p)>0$ for all $p\in P_{2d-2}(\R^n)\setminus\{0\}$. Then there exists a quadrature rule
$w\colon N\to\R_{>0}$ for $L$ minimizing \[\sum_{x\in N}w(x)\sum_{i=1}^nx_i^{2d}\]
among all quadrature rules for $L$. Any such $w$ satisfies
\[\#N\le\rh(n,2d)\le(2d-1)^n.\]
\end{thm}

\begin{proof}
Set $f:=X_1^{2d}+\ldots+X_n^{2d}$. By Theorem \ref{ct}, $(P_{L,f,\R^n})$ possesses an optimal solution $\La$ and for each such $\La$
there is an $h\in P_{2d}(\R^n)$ with $\La(h)=0$ and $h_{2d}=f$. Now choose an optimal solution $\La$ of $(P_{L,f,\R^n})$.
By Proposition \ref{primsol}(b), $\La$ possesses a quadrature rule $w\colon N\to\R_{>0}$ which is then in particular a quadrature rule
for $L=\La|_{\R[\x]_{2d-1}}$. Let $w'\colon N'\to\R_{>0}$ be another quadrature rule for $L$. Then
$\La'\colon\R[\x]_{2d}\to\R,\ p\mapsto\sum_{x\in N'}w'(x)p(x)$ is a feasible solution for $(P_{L,f,\R^n})$ and therefore
\begin{equation}\tag{$*$}
\sum_{x\in N}w(x)\sum_{i=1}^nx_i^{2d}=\La(f)\le\La'(f)=\sum_{x\in N'}w'(x)\sum_{i=1}^nx_i^{2d}
\end{equation}
by the optimality of $\La$. Finally suppose that $(*)$ holds with equality. It then remains to show that $\#N'\le\rh(n,2d)$.
Now $\La'$ is also an optimal solution if $(P_{L,f,\R^n})$. Hence there exists
$h\in P_{2d}(\R^n)$ with  $\La'(h)=0$ and $h_{2d}=f$. According to our strategy \ref{strategy}, we have therefore $N'\subseteq Z(h)$.
Consequently, $\#N'\le\#Z(h)\le\rh(n,2d)$ by Definition \ref{rhodef}.
\end{proof}

\noindent
Note that in the above theorem, the upper bound $(2d-1)^n$ will in most cases be much worse than the Carathéodory bound
discussed after Theorem \ref{bavaria} which would amount to
\[\binom{n+2d-1}n=\frac{(n+2d-1)\dotsm(1+2d-1)}{n!}.\]
due to the factorial appearing in the denominator. However, it is still interesting that we can find a quadrature rule minimizing the
indicated term with at most $(2d-1)^n$ nodes and most notably the actual bound $\rh(n,2d)$ is probably in most cases
much better than $(2d-1)^n$ even though it is rarely known (see for example Theorem \ref{petro} below).

\begin{cor}\label{gencor}
Given $d\in\N$ and $L\in M_{2d-1}(\R^n)$ with $L(p)>0$ for all $p\in P_{2d-2}(\R^n)\setminus\{0\}$, there exists a quadrature rule for $L$
with at most $\rh(n,2d)$ many nodes.
\end{cor}

\noindent
We do not know whether Corollary \ref{gencor} stays true without the technical
hypothesis that $L(p)>0$ for all $p\in P_{2d-2}(\R^n)\setminus\{0\}$.
On the one hand, the presence of a $p\in P_{2d-2}(\R^n)\setminus\{0\}$ with $L(p)=0$ obviously reduces the search space for the potential
quadrature rules dramatically since it obviously forces all nodes to be contained in $Z(p)$. On the other hand, it reduces the dimension of
the quadrature problem since $Z(p)$ has dimension smaller than $n$ (this follows for example by applying
\cite[Proposition 2.8.13]{bcr} to the complement of $Z(p)$). However, the geometry of $Z(p)$
might be much more complicated
than that of $\R^n$. In Theorem \ref{degreethree} below we will exploit a lucky situation where $Z(p)$ is an affine subspace.
The case where $Z(p)$ is (the real part of) a plane algebraic curve will be treated in Section \ref{plane} and {resorts} to some machinery
from algebraic geometry used in Section~\ref{curve}. Cases with more complicated $Z(p)$ will probably require much more algebraic geometry
and should be subject to future work. Anyway the following corollary demonstrates that in many cases the technical hypothesis is not a problem:

\begin{cor}
Suppose $d\in\N$ and
$\mu$ is a measure on $\R^n$ having finite moments up to degree $2d-1$. Moreover, suppose that the support of $\mu$ has
positive Lebesgue measure. Then there exists a quadrature rule of degree $2d-1$ for $\mu$ with at most $\rh(n,2d)$ many nodes.
\end{cor}

\begin{proof}
By Corollary \ref{gencor} and Definition \ref{dfqr}, it is enough to show that $\int p~d\mu>0$ for all
$p\in P_{2d-2}(\R^n)\setminus\{0\}$.
So let $p\in P_{2d-2}(\R^n)\setminus\{0\}$. In order to show $\int p~d\mu>0$, consider $S:=\R^n\setminus Z(p)$ and the measure
$\mu_S$ defined by $\mu_S(A):=\mu(A\cap S)$ for all Borel-measurable sets $A\subseteq\R^n$.
By the Bayer-Teichmann Theorem \ref{bavaria}, this measure has a quadrature rule $w\colon N\to\R_{>0}$ of degree $2d-1$ (we will need
only degree $2d-2$) with $N\subseteq S$. We have
\[\int p~d\mu=\int_S p~d\mu=\int p~d\mu_S=\sum_{x\in N}w(x)p(x)
\]
which is clearly positive once we know that $N\ne\emptyset$. Moreover, we have
\[\mu(S)=\int_S 1~d\mu=\int 1~d\mu_S=\sum_{x\in N}w(x)
\]
so that it is enough to show $\mu(S)>0$, i.e, the support of $\mu$ is not contained in $Z(p)$. To show this, it is enough to show that
$Z(p)$ has Lebesgue measure zero. But real zero sets $Z(h)$ of polynomials $h\in\R[\x]\setminus\{0\}$ are well-known to have zero
Lebesgue measure, see, e.g., \cite{ct} for a short proof of this.
\end{proof}

\begin{df}\label{bdef}
For any $h\in\R[\x]_{=2d}$, we denote by \[Z_{\R\P^{n-1}}(h):=\{x\in\R\P^{n-1}\mid h(x)=0\}\] its \emph{projective real zero set}.
For $n,d\in\N$, Choi, Lam and Reznick \cite[Page 2, Section 4]{clr} define
\[B_{n,2d}:=\sup\{\#Z_{\R\P^n}(h)\mid h\in P_{=2d}(\R\P^n),\ \#Z_{\R\P^n}(h)<\infty\}\in\N_0\cup\{\infty\}.
\]
\end{df}

\begin{prop}\label{rhole}
Let $n,d\in\N$. Then $\rh(n,2d)\le B_{n+1,2d}$.
\end{prop}

\begin{proof}
Let $h\in P_{2d}(\R^n)$ such that $h_{2d}=X_1^{2d}+\dots+X_n^{2d}$. It is enough to show that there is
$g\in P_{2d}(\R\P^n)$ such that $\#Z(h)\le\#Z_{\R\P^n}(g)<\infty$. This is easy: Take the homogenization
$g:=h_{[2d]}\in\R[X_0,\x]$. Then $Z_{\R\P^n}(g)=Z(h)\cup Z_{\R\P^{n-1}}(h_{2d})=Z(h)\cup \emptyset=Z(h)$
and therefore \[\#Z_{\R\P^n}(g)=\#Z(h)\le\rh(n,2d)\le(2d-1)^n<\infty.\]
\end{proof}

\noindent
We shortly resume everything what seems to be known about the numbers $B_{n,2d}$, see \cite[Theorem 4.3, Proposition 4.13 and
Page 13]{clr}:
\[d^2\le B_{3,2d}\le\frac32d(d-1)+1\]
where the upper bound is proven using Petrovsky's nontrivial bound
on the number of empty ovals of a non-singular plane algebraic curve of degree $2d$ \cite{pet} and is known to be sharp only for
$d\le3$,
$B_{3,6d}\ge 10d^2$, $B_{3,6d+2}\ge 10d^2+1$,
$B_{3,6d+4}\ge 10d^2+4$,
$\lim_{d\to\infty}\frac{B_{3,2d}}{d^2}$ exists and lies in the interval $[\frac{10}9,\frac 32]$ and
\[B_{4,4}=10\]
where the last equality is proven using a nontrivial result of Petrovsky and Oleinik on non-singular surfaces \cite{po}.
The consequences on $\rh(n,2d)$ we get by Proposition \ref{rhole} are extremely important for us and we will call them
the Petrovsky inequality or Petrovsky-Oleinik inequality, respectively:

\begin{thm}[Petrovsky inequality]\label{petro}
\[d^2\le\rh(2,2d)\le\frac32d(d-1)+1\text{ for $d\in\N$}\]
\end{thm}

\begin{proof}
The upper bound follows from \cite{pet}, see \cite[Page 13]{clr}. The lower bound is the trivial one from Remark \ref{roughlowerbound}.
\end{proof}

\noindent
We will exploit the Petrovsky inequality extensively in Section \ref{plane} below. Note that it is obviously sharp for $d\in\{1,2\}$.

\begin{thm}[Petrovsky-Oleinik inequality]\label{petroo}
\[8\le\rh(3,4)\le10\]
\end{thm}

\begin{proof}
The upper bound follows from \cite{po}, see \cite[Page 13]{clr}. The lower bound is again the trivial one from Remark \ref{roughlowerbound}.
\end{proof}

\noindent
Before we draw some consequences from the Petrovsky-Oleinik inequality, we have to introduce a quite obvious technique which will allow
us to reduce the number of variables $n$ whenever we search for a quadrature rule for a measure whose support is known to lie in an
affine subspace.

\bigskip\noindent
We define an \emph{affine subspace} of a vector space to be the empty set or a translated vector subspace (which is unique and called
its \emph{direction}). The \emph{dimension} of an affine subspace is the dimension of its direction unless it is empty in which case we
set it to $-1$.
Let now $U$ and $V$ be
affine subspaces of $\R^n$ and $\R^m$, respectively. A map $\ph\colon U\to V$ is \emph{affine linear} if there exists
a matrix $A\in\R^{m\times n}$ and a vector $b\in\R^m$ such that $\ph(x)=Ax+b$ for all $x\in\R^n$. Now let $\ph\colon U\to V$ be a
affine linear map. For each $L\in\R[\x]_d^*$ having a quadrature rule $w\colon N\to\R_{>0}$ with $N\subseteq U$, the
linear form
\[\ph(L)\colon\R[Y_1,\dots,Y_m]_d\to\R,\ p\mapsto\sum_{x\in N}w(x)p(\ph(x))\]
has the quadrature rule
\[\ph(w)\colon\ph(N)\to\R_{>0},\ y\mapsto\sum_{x\in N\cap\ph^{-1}(\{y\})}w(x)\]
and depends only on $\ph$ and $L$ but not on $w$ because for any
$A\in\R^{m\times n}$ and $b\in\R^m$ satisfying $\ph(x)=Ax+b$ for all $x\in\R^n$ we have
\[
\sum_{x\in N}w(x)p(\ph(x))=\sum_{x\in N}w(x)p(Ax+b)=L\left(p\left(A\begin{pmatrix}X_1\\\vdots\\X_n\end{pmatrix}+b\right)\right)
\]
for all $p\in\R[Y_1,\dots,Y_m]_d$ (this gives also an easy way of computing $\ph(L)$ without having a quadrature rule $w$ for $L$ at hand).

\bigskip\noindent
Now if $U$, $V$ and $W$ are affine subspaces of $\R^n$, $\R^m$ and $\R^\ell$, respectively, and $\ph\colon U\to V$ and
$\ps\colon V\to W$ are affine linear maps, then we have for
all $L\in\R[\x]_d^*$ and quadrature rules $w\colon N\to\R_{>0}$ with $N\subseteq U$ that
$\ps(\ph(w))=(\ps\circ\ph)(w)$ and consequently $\ps(\ph(L))=(\ps\circ\ph)(L)$.

\begin{lem}\label{redu}
Let $m,n,d,k\in\N$ and $U$ be an $n$-dimensional affine subspace of $\R^m$. Then the following are equivalent:
\begin{enumerate}[(a)]
\item Every element from $M_d(\R^n)$ possesses a quadrature rule with at most $k$ nodes.
\item Every element from $M_d(U)$ possesses a quadrature rule with at most $k$ nodes which are all contained in $U$.
\end{enumerate}
\end{lem}

\begin{proof}
Fix an affine linear bijection $\ph\colon\R^m\to U$.
Suppose that (a) holds. To show (b), let $L\in M_d(U)$. Since there exists a quadrature rule for $L$ with all nodes contained in $U$,
we can form $\ph^{-1}(L)\in M_d(\R^n)$ and by (a) we get a quadrature rule $w$ for $\ph^{-1}(L)$ with at most $k$ nodes.
Now $\ph(w)$ is a quadrature rule for $\ph(\ph^{-1}(L))=(\ph\circ\ph^{-1})(L)=\id_{\R^m}(L)=L$ with at most $k$ nodes.
This shows (a)$\implies$(b). The proof of (b)$\implies$(a) is analogous.
\end{proof}

\noindent
The following is well-known folklore:

\begin{prop}\label{sos2}
For any $p\in P_2(\R^n)$, there are $\ell_0,\dots,\ell_n\in\R[\x]_1$ such that \[p=\sum_{i=0}^n\ell_i^2.\]
\end{prop}

\begin{proof}
Let $v$ be the column vector with entries $1,X_1,\dots,X_n$. Then there is a (unique) symmetric matrix $G\in\R^{(n+1)\times(n+1)}$ such
that $p=v^TGv$. From the nonnegativity of $p$ on $\R^2$, it follows by continuity that $G$ is positive semidefinite and therefore can
be factored (e.g., by the spectral theorem) as $G=L^TL$ for some matrix $L\in\R^{(n+1)\times(n+1)}$ (take for example the Cholesky
factorization). Now $p=v^TL^TLv=(Lv)^T(Lv)=\sum_{i=0}^n\ell_i^2$ if $\ell_0,\dots,\ell_n$ are the entries of $Lv$.
\end{proof}

\begin{thm}\label{degreethree}
For every $L\in M_3(\R^n)$, there exists a quadrature rule for $L$ with at most $\rh(n,4)$ nodes.
\end{thm}

\begin{proof}
We proceed by induction on $n$. For $n=1$, we get a quadrature rule for $L$ with at most $2$ nodes by Example \ref{uniex} and
$2\le\rh(n,4)$ by Remark \ref{roughlowerbound}. In the induction step from $n-1$ to $n$ ($n\ge2$), we get the result directly
from Corollary \ref{gencor} except if there exists $p\in P_2(\R^n)\setminus\{0\}$ with $L(p)=0$. So fix
$p\in P_2(\R^n)\setminus\{0\}$ with $L(p)=0$. We will now use the induction hypothesis. Since any quadrature rule for $L$ has all its nodes
contained in $Z(p)$, we have $L\in M_3(Z(p))$. Proposition \ref{sos2} yields $\ell_0,\dots,\ell_n\in\R[\x]_1$ with $p=\sum_{i=0}^n\ell_i^2$.
Since $p\ne0$, we have without loss of generality $\ell_0\ne0$. Now $U:=Z(\ell_0)$ is an affine subspace of $\R^n$ of dimension $n-1$.
By induction hypothesis and by Lemma \ref{redu}, there exists a quadrature rule for $L$ with at most $\rh(n-1,4)$ nodes.
It remains to show that $\rh(n-1,4)\le\rh(n,4)$ but this follows easily from the fact that for each $h\in P_4(\R^{n-1})$ we have
$h+X_n^{2d}\in P_4(\R^n)$ and $Z(h+X_n^{2d})=Z(h)\times\{0\}$.
\end{proof}

\begin{cor}\label{oleinik}
Suppose $L\in M_3(\R^3)$. Then there exists a quadrature rule for $L$ with at most $\rh(3,4)\le10$ many nodes.
\end{cor}

\begin{proof}
Combine Theorem \ref{degreethree} with the Petrovsky-Oleinik inequality \ref{petroo}.
\end{proof}


\section{Quadrature on plane algebraic curves}\label{curve}

\begin{rem}\label{dubstep}
If $p\in\R[X,Y]$ and $x$ is an isolated point of $Z(p)$ in $\R^2$, then $x$ is double root of $p$ (cf. Definition \ref{doubleroot}). This follows immediately from the most basic version of the implicit function theorem.
\end{rem}

\begin{rem}\label{irred}
Recall that two polynomials in $\R[\x]$ are \emph{associate} if one is the other multiplied by a non-zero real constant.
The same definition applies to the ring of complex polynomials $\C[\x]$ if one allows a non-zero complex constant.
If $p$ is irreducible in $\R[\x]$, then
\begin{itemize}
\item either it is also irreducible in $\C[\x]$
\item or there is an irreducible element $q$ of $\C[\x]$ such that $q$ is not associate to its complex conjugate polynomial
$q^*$ in $\C[\x]$ and $p$ is associate to $q^*q$ in $\R[\x]$.
\end{itemize}
Indeed, choose an irreducible factor $q$ of $p$ in $\C[\x]$.
Then $q^*$ is also an irreducible factor of $p$ in $\C[\x]$. First consider the case where $q$ and $q^*$ are associate in $\C[\x]$, i.e.,
there is $\ze\in\C$ such that $q=\ze q^*$. Then $q^*=\ze^*q$ and therefore $q=\ze\ze^*q$ from which follows $|\ze|=1$. Choose
$\xi\in\C$ such that $\xi^2=\ze$. Now $\xi^*\xi=|\xi|^2=|\ze|=1$ and therefore $(\xi^*q)^*=\xi q^*=\xi(\xi^2)^*q=\xi\xi^*\xi^*q=\xi^*q$.
Hence $\xi^*q\in\R[\x]$.
Because $\xi^*q\in\R[\x]$ divides $p$ in $\C[\x]$, it divides $p$ also in $\R[\x]$ since a real system of linear equations has a real solution if
and only if it has a complex solution. By the irreducibility of $p$ in $\R[\x]$, we have that $p$ and $\xi^*q$ are associate in $\R[\x]$ and therefore
also in $\C[\x]$. Since $q$ is irreducible in $\C[\x]$, we have that $p$ is also irreducible in $\C[\x]$.
Finally consider the other case where $q$ and $q^*$ are not associate in $\C[\x]$. Then $q^*q\in\R[\x]$ divides $p$ in $\C[\x]$ and therefore
also in $\R[\x]$. By the irreducibility of $p$ in $\R[\x]$, $p$ and $q^*q$ are associate in $\R[\x]$.
\end{rem}

\noindent
The proof of the theorem will rely essentially on the following variant of Bézout's theorem which we did not find in the literature.

\begin{thm}\label{newbezout}
Let $g,h\in\R[X,Y]\setminus\{0\}$ such that $Z(g)\cap Z(h)$ is finite and $h(x)\ge0$ for all $x\in Z(g)$.
Then $2\#(Z(g)\cap Z(h))\le(\deg g)(\deg h)$.
\end{thm}

\begin{proof}
One easily reduces to the case where $g$ is irreducible in the ring $\R[X,Y]$. First consider the case where $g$ divides $h$ in $\R[X,Y]$.
Then $Z(g)=Z(g)\cap Z(h)$ is finite and therefore every real root of $g$ is isolated in $\R^2$ and hence double by Remark \ref{dubstep}.
But then we can exchange $h$ by an arbitrary non-zero linear combination $\tilde h$
of the partial derivatives of $g$ with respect to $X$ and $Y$ since $Z(g)\cap Z(\tilde h)=Z(g)$ is finite, $\tilde h(x)=0\ge0$ for all
$x\in Z(g)$ and $\deg\tilde h\le1+\deg\tilde h\le\deg g\le\deg h$. In this way we can from now on
assume that $g$ does not divide $h$ in $\R[X,Y]$. It then follows easily from Remark \ref{irred}
that $g$ and $h$ do not share any common irreducible factor in $\C[X,Y]$, i.e., the hypotheses for
Bézout's theorem for not necessarily reduced \cite[Section 5.1]{ful} affine algebraic curves are fulfilled which says that
\[\sum_{(x,y)\in V}\mu_{(x,y)}(g,h)\le(\deg g)(\deg h).\]
where
\[V:=\{(x,y)\in\C^2\mid g(x,y)=h(x,y)=0\}\]
and $\mu_{(x,y)}(g,h)\in\N_0$ denotes for each $(x,y)\in\C^2$ the intersection multiplicity of $g$ and $h$ at the point $(x,y)$
\cite[Section 3.3]{ful}. A more precise projective version is given in \cite[Section 5.3]{ful} from which follows obviously the just applied
affine version by \cite[top of Page 54]{ful}. Since $Z(g)\cap Z(h)=\R^2\cap V\subseteq V$, it follows that
\[\sum_{(x,y)\in Z(g)\cap Z(h)}\mu_{(x,y)}(g,h)\le(\deg g)(\deg h).\]
Fix now $(x,y)\in Z(g)\cap Z(h)$. It is enough to show that $\mu_{(x,y)}(g,h)\ge 2$. Assume this is not true. Then $\mu_{(x,y)}(g,h)=1$
by \cite[3.3(2)]{ful}. Now $(x,y)$ is neither a double root of $g$ nor of $h$ and the evaluations of
$\left(\frac{\partial g}{\partial X} \frac{\partial g}{\partial Y}\right)$ and $\left(\frac{\partial h}{\partial X} \frac{\partial h}{\partial Y}\right)$
at $(x,y)$ are not collinear by \cite[3.3(5)]{ful} and therefore after an affine change of coordinates \cite[3.3(3)]{ful} without loss of generality
$(0,1)$ and $(1,0)$. By the most basic version of the implicit function theorem, there is an open neighborhood $U$ of $x$ in $\R$ such that
there is a differentiable function $f\colon U\to\R$ satisfying $f(x)=y$ and $g(\xi,f(\xi))=0$ for all $\xi\in U$. By the chain rule,
the derivative of the function $U\to\R,\ \xi\mapsto h(\xi,f(\xi))$ at $x$ is $1$, i.e., $\frac{h(\xi,f(\xi))}\xi$ is close to $1$ and therefore
$h(\xi,f(\xi))<0$ when $\xi\in U$ is negative and sufficiently close to $0$. This contradicts $h(\xi,f(\xi))\ge0$ for all $\xi\in U$.
\end{proof}

\begin{cor}\label{kundk}
Let $p\in\R[X,Y]$ be of degree $k\in\N_0$ such that $Z(p)$ is finite. Then \[2\#Z(p)\le k(k-1).\]
\end{cor}

\begin{proof}
Obviously we have $k=\deg p\ge 1$. Therefore at least one of the partial derivatives of $p$ is not the zero polynomial. Without loss of
generality $q:=\frac{\partial p}{\partial X}\ne0$. By Remark \ref{dubstep} we have $q(x)=0$ for all $x\in Z(p)$. Therefore
$Z(p)=Z(p)\cap Z(q)$ and $2\#Z(p)=2\#(Z(g)\cap Z(q))\le(\deg p)(\deg q)=k(k-1)$ by Theorem \ref{newbezout}.
\end{proof}

\noindent
We come now to the main result of this section.

\begin{thm}\label{curvethm}
Let $d\in\N$, $g\in\R[X,Y]\setminus\{0\}$, $k:=\deg g$, $S:=Z(g)$ and $L\in M_{2d-1}(S)$.
Suppose $f\in\R[X,Y]_{2d}$ satisfies $f_{2d}(x,y)>0$ for all $(x,y)\in S_\infty$.
Then there exists a quadrature rule $w\colon N\to\R_{>0}$ for $L$ with $N\subseteq S$ and
\[\#N\le dk\] that minimizes
\[\sum_{(x,y)\in N}w(x,y)f(x,y)\]
among all quadrature rules for $L$ whose nodes are all contained in $S$.
\end{thm}

\begin{proof}
By Proposition \ref{primsol}, we can choose an optimal solution $\La_0$ of $(P_{L,f,S})$ and a quadrature rule $w_0\colon N_0\to\R_{>0}$ with
$N_0\subseteq S$. Using a Carathéodory argument similar to the one in the discussion after Theorem \ref{bavaria} in the vector space
$(\R[X,Y]_{2d-1}+\R f)^*$, we see that we can assume $\#N_0\le\dim(\R[X,Y]_{2d-1}+\R f)^*=\dim(\R[X,Y]_{2d-1}+\R f)=
(\dim\R[X,Y]_{2d-1})+1=(2d+1)d+1=2d^2+d+1$. If $k\ge 2d+2$, we have thus $\#N_0\le2d^2+d+1\le dk$ and we are done. Therefore
suppose from now on that we are in the nontrivial case where $k\le2d+1$.

Since $\R[X,Y]$ is a factorial ring, we can choose $m\in\N_0$ and irreducible $p_1,\dots,p_m\in\R[X,Y]_{2d}$ such
that \[g=p_1\dotsm p_m.\] Setting $S_i:=Z(p_i)$ and $k_i:=\deg p_i$ for $i\in\{1,\dots,m\}$, it follows that \[S=S_1\cup\ldots\cup S_m.\]
Choose pairwise distinct subsets $N_{0,1},\dots,N_{0,m}$ of $N_0$ such that
$N_{0,i}\subseteq S_i$ for all $i\in\{1,\dots,m\}$ and
\[N_0=N_{0,1}\dotcup\ldots\dotcup N_{0,m}\] and let
$\La_{0,i}\in M_{2d}(S_i)$ for $i\in\{1,\dots,m\}$ be defined by the property that $w|_{N_i}$ is a quadrature rule of $\La_{0,i}$.
We have that $\La_{0,i}$ is a feasible solution of $(P_{L_i,f,S_i})$ where $L_i:=\La_{0,i}|_{\R[\x]_{2d-1}}$.
Then \[\La_0=\La_{0,1}+\dots+\La_{0,m}.\]
Next observe the following important fact: If $i\in\{1,\dots,m\}$ such that $S_i$ is finite, then
\[\#N_{i,0}\le\#S_i\le\frac12k_i(k_i-1)\le k_id=dk_i\]
where the first inequality follows from $N_{i,0}\subseteq S_i$, the second from Corollary \ref{kundk} and the third from
$k_i\le k\le 2d+1$. For $i\in\{1,\dots,m\}$,
\begin{itemize}
\item in case that $\#N_{i,0}\le dk_i$ set  $\La_i:=\La_{i,0}$, $N_i:=N_{i,0}$ and $w_i:=w_{i,0}$,
\item in case that $\#N_{i,0}>dk_i$ choose by Theorem \ref{qt}(a) $\La_i$ as an optimal solution of $(P_{L_i,f,S_i})$ on the relative
boundary of $P_{2d}(S_i)$ (this is possible since $S_i$ is infinite) and choose a quadrature rule $w_i\colon N\to\R_{>0}$ with
$N_i\subseteq S_i$.
\end{itemize}
Now \[\La:=\La_1+\ldots+\La_m\] is obviously
a feasible solution of $(P_{L,f,S})$ and therefore $\La_0(f)\le\La(f)$ by the optimality of $\La_0$. On the other hand
$\La_i(f)\le\La_{i,0}(f)$ by the choice of $\La_i$ for each $i\in\{1,\dots,m\}$. Summing up over $i$ on both sides, this yields
$\La(f)\le\La_0(f)$. Therefore $\La(f)=\La_0(f)$ and $\La$ is an optimal solution of $(P_{L,f,S})$.
Now
\[w\colon N:=N_1\cup\ldots\cup N_m\to\R_{>0},\ (x,y)\mapsto\sum_{\substack{i=1\\(x,y)\in N_i}}^mw_i(x,y)\]
is a quadrature rule for $\La$ and in particular for $L$ with $N\subseteq S$.
It has the claimed optimality property by the optimality
of $\La$. It only remains to shows that $\#N\le dk$.
To this end, it suffices to show that $\#N_i\le dk_i$ for
$i\in\{1,\dots,m\}$. Fix $i\in\{1,\dots,m\}$. By choice of $\La_i$ there is nothing to show if  $\#N_{i,0}\le dk_i$.
So suppose that $\#N_{i,0}>dk_i$. Then
$\La_i$ lies on the relative boundary of $P_{2d}(S_i)^*$ and we get by Theorem \ref{qt}(b) a polynomial
$h\in P_{2d}(S_i)\setminus V_{2d}(S)$ with $\La_i(h)=0$. The irreducible affine $\R$-variety
\[V:=\{x\in\C^n\mid p_i(x)=0\}\] has dimension $1$ \cite[Proposition 4.4(f)]{kun}. Its affine $\R$-subvariety
\[\{(x,y)\in\C^2\mid p_i(x,y)=0,h(x,y)=0\}\ne V\]
has dimension at most $0$ and therefore is finite \cite[Proposition 3.11(f)]{kun}.
In particular, $S_i\cap Z(h)$ is finite and Theorem \ref{newbezout} shows
\[2\#(S_i\cap Z(h))\le k_i(2d).\]
Since $N_i\subseteq S_i\cap Z(h)$, it follows that $\#N_i\le dk_i$ as desired.
\end{proof}

\begin{cor}\label{curvecor}
Let $d\in\N$, $g\in\R[X,Y]\setminus\{0\}$, $k:=\deg g$, $S:=Z(g)$ and $L\in M_{2d-1}(S)$.
Then there exists a quadrature rule for $L$ with at most $dk$ nodes and all nodes contained in $S$.
\end{cor}

\noindent
In the situation of the above corollary, the trivial Carathéodory bound discussed after the Bayer-Teichmann theorem \ref{bavaria} would
yield the same statement with $\#N\le dk$ replaced by $\#N\le\binom{2+(2d-1)}2=d(2d+1)$. It is therefore not interesting for $k>2d$.
However, it yields already for small $k$ very interesting results:
The first example concerns the Gaussian quadrature on the line from Example \ref{uniex}. It will be closely investigated in Section \ref{line}
below.

\begin{ex}[Gauss quadrature on the line]\label{gaussex}
Let $d\in\N$ and choose $g\in\R[X,Y]_1\setminus\{0\}$. Then $S$ is a straight line in $\R^2$.
By Lemma \ref{redu}, Corollary \ref{curvecor} says in this case therefore that every element from $M_{2d-1}(\R)$
possesses a quadrature rule with at most $d$ nodes.
\end{ex}

\noindent
The only other previously known nontrivial example of Corollary \ref{curvecor} which we are aware of is the odd degree case of
Szegő quadrature on a circle discovered by Jones, Njåstad and Thron only in 1989 \cite[Theorem 8.4]{jnt}.

\begin{ex}[Szegő-quadrature on the circle]\label{szegö}
Let $d\in\N$ and set $g:=X^2+Y^2-1$. Then $S$ is the unit circle in $\R^2$. Corollary \ref{curvecor} now says that
every element from $M_{2d-1}(S)$ possesses a quadrature rule with at most $2d$ nodes.
\end{ex}

\noindent
Let $S:=\{(x,y)\in\R^2\mid x^2+y^2=1\}$ denote the unit circle .
The functions \[S\to\R,\ (x,y)\mapsto x^\al y^\be\qquad(\al,\be\in\N_0,\,\al+\be\le 2d-1)\] can be written as
complex linear combination of the functions \[S\to\C,\ (x,y)\mapsto (x+\ii y)^k\qquad(k\in\{-(2d-1),\dots,2d-1\})\]
and vice versa where $\ii:=\sqrt{-1}\in\C$ denotes the imaginary unit. This shows that Example \ref{szegö} above
is really just a ``real formulation'' of the odd degree case of \cite[Theorem 8.4]{jnt}.


\section{Cubature on the plane}\label{plane}

\noindent
Recall that a polynomial $h\in\R[\x]$ is called \emph{square-free} if there is no irreducible polynomial $p\in\R[\x]$ such that $p^2$ divides
$h$ in $\R[\x]$. We begin with an easy lemma which must certainly well-know but for which we did not find a suitable reference.

\begin{lem}\label{quadratfrei}
Any square-free polynomial in $\R[X,Y]$ has only a finite number of double roots in $\C^2$.
\end{lem}

\begin{proof}
Let $h\in\R[X,Y]$ be square-free. Since $\R[X,Y]$ is a factorial ring, we can write
$h=cp_1\dotsm p_m$ for some $c\in\R^\times$, $m\in\N_0$ and irreducible pairwise non-associate $p_1,\dots,p_m\in\R[X,Y]$.
Fix $i\in\{1,\dots,m\}$. It is enough to show that the set of common complex zeros of $p_i$ and both partial derivatives
of $h$ is finite. Since $p_i$ defines the irreducible curve $C:=\{(x,y)\in\C^2\mid p_i(x,y)=0\}$,
it is enough to show that not both partial derivatives of $h$ vanish simultaneously on this curve.

Suppose for the moment that $\frac{\partial h}{\partial X}$ vanishes on $C$. We claim that then $\frac{\partial p_i}{\partial X}=0$.
Indeed, $p_i$ divides a power of $\frac{\partial h}{\partial X}$ in $\R[X,Y]$ by Hilbert's Nullstellensatz \cite[Chapter 1, §3, Proposition 3.7]{kun}
and therefore also $\frac{\partial h}{\partial X}$ itself since $p_i$ is irreducible.
Because $p_i$ now divides the left hand side as well as all but the $i$-th term in the sum on the right hand side of the equation
\[\frac{\partial h}{\partial X}=\sum_{j=1}^m\frac{\partial p_j}{\partial X}\prod_{k\ne j}p_k,\]
it also divides $\frac{\partial p_i}{\partial X}\prod_{k\ne i}p_k$. Because $p_1,\dots,p_m$ are pairwise non-associate,
it follows that $p_i$ divides $\frac{\partial p_i}{\partial X}$ which, for degree reasons, is only possible if
$\frac{\partial p_i}{\partial X}=0$.

If now both partial derivatives of $h$ vanished on $C$, then both partial derivatives of $p_i$ would be the zero polynomial, i.e., $p_i$ would
be a constant polynomial which is not possible by definition of an irreducible polynomial.
\end{proof}

\noindent
The following is a variant of \cite[Section 5.4, Theorem 2]{ful}.

\begin{prop}\label{hestim}
Let $h\in\R[X,Y]$ be square-free of degree $\ell$ and $h(x)\ge0$ for all $x\in\R^n$. Then $2\#Z(h)\le\ell(\ell-1)$.
\end{prop}

\begin{proof}
This is clear if $\ell=0$. So suppose $\ell>0$. Then at least one of both partial derivatives of $h$ is
not the zero polynomial, say $g:=\frac{\partial h}{\partial X}\ne0$. By Remark \ref{nonn}, $g(x)=0\ge0$ for all $x\in Z(h)$. If
$Z(h)\cap Z(g)=Z(h)$ is finite, we can conclude by Theorem \ref{newbezout} (with the roles of $g$ and $h$ interchanged).
Therefore it remains only to show that $Z(h)$ is finite. But this follows from Remark~\ref{nonn} together with Lemma \ref{quadratfrei}.
\end{proof}

\noindent
We are now ready to state the main result of this section.

\begin{thm}\label{planethm}
Let $d\in\N$ and $L\in M_{2d-1}(\R^2)$.
Then there exists a cubature rule $w\colon N\to\R_{>0}$ for $L$ with
\[\#N\le\rh(2,2d)\le\frac32d(d-1)+1\] that minimizes
\[\sum_{(x,y)\in N}w(x,y)(x^{2d}+y^{2d})\]
among all cubature rules for $L$.
\end{thm}

\begin{proof}
If $L(p)>0$ for all $p\in P_{2d-2}(\R^2)\setminus\{0\}$, then we conclude by Theorem~\ref{genthm} and the Petrovsky inequality \ref{petro}.
Therefore suppose now that we have fixed $p\in P_{2d-2}(\R^2)\setminus\{0\}$ such that $L(p)=0$. Now write
\[p=g^2h\] with $g,h\in\R[X,Y]\setminus\{0\}$ such that $h$ is square-free.
We see that $h(x)\ge0$ for all $x\in\R^2\setminus Z(g)$ and by continuity even for all $x\in\R^2$.
Define $k,\ell\in\N_0$ by $k=\deg g$ and $2\ell=\deg h$ (note that the degree of $h$ is even since the degree of $p$ is even) so that
\[k+\ell\le d-1.\]
Using Proposition
\ref{primsol}, we can choose an optimal solution $\La_0$ of $(P_{L,X^{2d}+Y^{2d},\R^2})$ together with a quadrature rule
$w_0\colon N_0\to\R_{>0}$ for $\La_0$. Because of $L(p)=0$, we have $N_0\subseteq Z(p)=Z(g)\cup Z(h)$. 
Now set $N_1:=N_0\cap Z(g)$ and $N_2:=N_0\setminus Z(g)\subseteq Z(h)$. For $i\in\{1,2\}$, we set now $w_i:=w_0|_{N_i}$ and denote by
$L_i\in\R[X,Y]_{2d-1}^*$ the linear form having $w_i$ as quadrature rule.
We have $L_1\in M_{2d-1}(Z(g))$ and $\#N_2\le\#Z(h)\le
\frac{\ell(\ell-1)}2$ by Proposition \ref{hestim}. Now we apply Theorem \ref{curvethm} to see that there exists a quadrature rule
$w'\colon N'\to\R_{>0}$ for $L_1$ with $N'\subseteq Z(g)$ such that $\#N'\le dk$
that minimizes
\[\sum_{(x,y)\in N'}w'(x,y)(x^{2d}+y^{2d})\]
among all quadrature rules for $L_1$ whose nodes are all contained in $Z(g)$. In particular,
\begin{equation}\tag{$*$}\label{mineq}
\sum_{(x,y)\in N'}w'(x,y)(x^{2d}+y^{2d})\le\sum_{(x,y)\in N_1}w_1(x,y)(x^{2d}+y^{2d}).
\end{equation}
Now finally define $w\colon N:=N'\dotcup N_2\to\R_{>0}$ by $w|_{N'}=w'$ and $w|_{N_2}=w_2$. Then $w$ is a quadrature rule
for $L_1+L_2=L$ and \[\#N=\#N'+\#N_2\le dk+\frac{\ell(\ell-1)}2.\]
Adding the same term to both sides of \eqref{mineq}, we obtain
\[
\sum_{(x,y)\in N}w(x,y)(x^{2d}+y^{2d})\le\sum_{(x,y)\in N_0}w_0(x,y)(x^{2d}+y^{2d}).
\]
By the optimality of $\La_0$, $w_0$ minimizes the right hand side in this inequality among all quadrature rules for $L$.
Therefore $w$ minimizes the left hand side (which now turns out to be equal to the right hand side)
of this inequality among all quadrature rules for $L$. It remains only to show that
\[dk+\frac{\ell(\ell-1)}2\le\rh(2,2d).\]
By the trivial part of the Petrovsky inequality \ref{petro}, it suffices to show 
\[2dk+\ell(\ell-1)\le2d^2,\]
or equivalently
\[\ell(\ell-1)\le2d(d-k).\]
Using $d-k\ge\ell+1\ge\ell$, we reduce this to
\[\ell(\ell-1)\le2d\ell\]
which is trivially true.
\end{proof}

\begin{cor}\label{planecor}
For all $d\in\N$ and $L\in M_{2d-1}(\R^2)$, there
exists a cubature rule for $L$ with at most $\rh(2,2d)\le\frac32d(d-1)+1$ many nodes.
\end{cor}

\noindent
We don't have any reason to expect the above upper bound of $\frac32d(d-1)+1$ to be optimal but we 
conclude the section by showing that at least it cannot be improved strictly below $(d-1)^2$. The idea underlying
this is folklore and contained in the following proposition.

\begin{prop}\label{evlu}
Let $d\in\N_0$ and $f\in P_{2d}(\R^n)$ such that that the family $(\ev_x)_{x\in Z(f)}$ is linearly independent in
$\R[\x]_{2d}^*$ (in particular, $Z(f)$ is finite) where
\[\ev_x\colon\R[\x]_{2d}\to\R,\ p\mapsto p(x)\] denotes the point evaluation in $x$ for each $x\in\R^n$.
Then for each $w\colon Z(f)\to\R_{>0}$, the linear form
\[L\colon\R[\x]_{2d}\to\R,\ p\mapsto\sum_{x\in Z(f)}w(x)p(x)\]
has exactly one quadrature rule, namely $w$.
\end{prop}

\begin{proof}
Let $u\colon N\to\R$ be any quadrature rule for $L$. We have to show $u=w$.
Since $L(f)=0$, it follows that $N\subseteq Z(f)$. Extend $u$ to
$u_0\colon Z(f)\to\R$ by setting it to zero outside of $N$. Then
$\sum_{x\in Z(f)}w(x)\ev_x=L=\sum_{x\in Z(f)}u_0(x)\ev_x$ which implies $w=u_0$ by the linear independence of $(\ev_x)_{x\in Z(f)}$.
In hindsight, we get now $u=u_0$ whence $u=w$.
\end{proof}

\begin{lem}\label{lemlu}
Let $d\in\N_0$. The family $(\ev_{(x,y)})_{(x,y)\in\{1,\dots,d\}^2}$ is linearly independent in $\R[\x]_{2d}^*$
where $\ev_{(x,y)}\colon\R[\x]_{2d}\to\R,\ p\mapsto p(x,y)$ denotes the point evaluation in $(x,y)$ for every
$(x,y)\in\{1,\dots,d\}^2$.
\end{lem}

\begin{proof}
It suffices obviously to show that any function $\{1,\dots,d\}^2\to\R$ can be interpolated by a polynomial $p\in\R[X,Y]_{2d}$.
Such a function can certainly by interpolated by a polynomial $p\in\R[X,Y]$ of some degree and adding any polynomial from
the ideal generated in $\R[X,Y]$
by $(X-1)\dotsm(X-d)$ and $(Y-1)\dotsm(Y-d)$ will not change the values of this polynomial on the grid
$\{1,\dots,d\}^2$. But modulo this ideal any polynomial is congruent to a linear combination of the monomials $X^iY^j$ with
$0\le i,j\le d$, in particular to a polynomial of degree at most $2d$ (in case $d>0$ one can even do with $0\le i,j\le d-1$ so that
one gets a sharper result but we do not need this).
\end{proof}

\begin{rem}\label{lowb}
Let $d\in\N$ and take \[f:=(X-1)^2\dotsm(X-(d-1))^2+(Y-1)^2\dotsm(Y-(d-1))^2\in P_{2d-2}(\R^n)\] so that
$Z(f)=\{1,\dots,d-1\}^2$ (cf. Remark \ref{roughlowerbound}). Choose $w\colon Z(f)\to\R_{>0}$ and define
\[L\colon\R[X]_{2d-1}\to\R,\ p\mapsto\sum_{x\in Z(f)}w(x)p(x).\]
By Proposition \ref{evlu} and Lemma \ref{lemlu}, $L$ (even $L|_{\R[X]_{2d-2}}$) has exactly one quadrature rule.
But this quadrature rule is $w$ and has
\[(d-1)^2\]
many nodes. Therefore the upper bound in Corollary \ref{planecor} cannot be improved from
$\frac32d(d-1)+1$ to something strictly below $(d-1)^2$.
\end{rem}

\noindent
One could ask about possible improvements of the upper bound $\frac32d(d-1)+1$ in Corollary \ref{planecor}
under the additional hypothesis $L(p)>0$ for all $p\in P_{2d-2}(\R^2)\setminus\{0\}$ from Theorem \ref{genthm}. In this case,
the idea of Remark \ref{lowb} does not work.
The following remark which we owe to di Dio and Schmüdgen \cite{dds} gives another lower bound
under this additional hypothesis.

\begin{rem}\label{dimcount}
$M_{2d-1}(\R^2)$ is a cone with nonempty interior (e.g., by Remark \ref{intdua})
in the vector space $\R[X,Y]_{2d-1}^*$ which has dimension $d(2d+1)$.
For $k\in\N_0$,
consider the semialgebraic set $A\subseteq M_{2d-1}(\R^2)$ consisting of all linear forms on $\R[X,Y]_{2d-1}^*$ that possess
a quadrature formula with at most $k$ nodes. Obviously, $A$ is the image of a polynomial (in particular semialgebraic)
map $\R^{3k}\to\R[X,Y]_{2d-1}^*$
and therefore is a semialgebraic set of dimension at most $3k$ \cite[Theorem~2.8.8]{bcr}. In order that
$A$ can have nonempty interior, we must have $3k\ge d(2d+1)$ by
\cite[Section~2.8]{bcr}, i.e., $k\ge\left\lceil\frac{d(2d+1)}3\right\rceil$.
This shows that the linear forms that possess a quadrature rule but do not possess one with less than
\[\left\lceil\frac{d(2d+1)}3\right\rceil\]
many nodes lie dense in $M_{2d-1}(\R^2)$.
\end{rem}

\noindent
{The problem with Remarks \ref{lowb} and
\ref{dimcount} is that they say little or nothing about the number of nodes for quadrature of
concrete measures one encounters in practice.}
Indeed, the only {nontrivial lower bounds for concrete measures} that seem to be known are
the results of Möller \cite{mol} and {Cools et al. \cite{vc,cs,hs}}. Möller gives for each $d\in\N$ 
a huge family of $L\in M_{2d-1}(\R^2)$
satisfying $L(p)>0$ for all $p\in P_{2d-2}(\R^2)\setminus\{0\}$ such that any quadrature rule for
$L$ needs at least 
\[\frac{d(d+1)}2+\left\lfloor\frac d2\right\rfloor\]
many nodes. For $d\ge5$, Möller's lower bound has been improved in \cite{vc,cs,hs}
for small subfamilies of his family by exactly one thus yielding
\[\frac{(d+1)^2}2.\]
Note that these results became an easy rank argument if one wanted to show only the lower bound of
\[\frac{d(d+1)}2\]
for $d\in\N$ which is usually attributed to Stroud.
Indeed, the additional hypothesis $L(p)>0$ for all $p\in P_{2d-2}(\R^2)\setminus\{0\}$ means that
$(L(\x^{\al+\be}))_{\al,\be\in\N_0^n,|\al|\le d-1,|\be|\le d-1},$ regarded as a symmetric matrix, is positive definite and thus has
full rank
\[\frac{d(d+1)}2=\binom{d-1+2}2=\dim\R[X,Y]_{d-1}.\]
Even if one allowed negative weights in quadrature rules, one needed at least as many nodes as the rank of this matrix indicates
by an easy exercise which we leave to the reader.
Also the just mentioned lower bounds of Möller and of Cools and Schmid continue to work if one allowed in a quadrature arbitrary real
weights instead of just positive ones.

\bigskip\noindent
We conclude this section with the following table that illustrates the quality of our upper bound from Corollary
\ref{planecor}.

{\small
\[
\begin{array}{ccc|cccccccccc}
&&d&1&2&3&4&5&6&7&8&9&10\\
&\begin{matrix}\text{\tiny Degree of}\\\text{\tiny exactness}\end{matrix}&2d-1&1&3&5&7&9&11&13&15&17&19\\
\hline
&&&\\
\text A&\begin{matrix}\text{\tiny Stroud's trivial}\\\text{\tiny lower bound}\end{matrix}&\frac{d(d+1)}2&
1&3&6&10&15&21&28&36&45&55\\
&&\\
\text B&\text{\tiny Möller's lower bound}&\frac{d(d+1)}2+\left\lfloor\frac d2\right\rfloor&1&4&7&12&17&24&31&40&49&60\\
&&\\
\text C&\begin{matrix}\text{\tiny Best known upper}\\\text{\tiny bound for the square}
\end{matrix}
&\blacksquare&1^{\rlap*}&4^{\rlap*}&7^{\rlap*}&12^{\rlap*}&17^{\rlap*}&24^{\rlap*}&33&44&56&68\\
&&\\
\text D&\begin{matrix}\text{\tiny Best known upper}\\\text{\tiny bound for the disk}
\end{matrix}
&\text{\LARGE$\bullet$}&1^{\rlap*}&4^{\rlap*}&7^{\rlap*}&12^{\rlap*}&18^{\rlap{*\tiny O}}&25^{\rlap{\tiny O}}&34^{\rlap{\tiny O}}&44&57&69^{\rlap{\tiny O}}\\
&&\\
\text E&\begin{matrix}\text{\tiny di Dio and Schmüdgen's}\\\text{\tiny lower bound}\\\text{\tiny from Remark \ref{dimcount}}
\\\text{\tiny\cite[Thms. 27 and 53]{dds}}
\end{matrix}
&\left\lceil\frac{d(2d+1)}3\right\rceil&1&4&7&12&19&26&35&46&57&70\\
&&\\
\text F&\begin{matrix}\text{\tiny Our lower bound}\\\text{\tiny from Remark \ref{lowb}}\end{matrix}
&(d-1)^2&0&1&4&9&16&25&36&49&64&81\\
&&\\
\text G&\begin{matrix}\text{\tiny Our upper bound}\\\text{\tiny from Corollary \ref{planecor}}\end{matrix}&\frac32d(d-1)+1&
10&4&10&19&31&46&64&85&109&136\\
&&\\
\text H&&d^2&1&4&9&16&25&36&49&64&81&100\\   
&&\\
\text I&\begin{matrix}\text{\tiny Carathéodory's trivial}\\\text{\tiny upper bound}\end{matrix}&d(2d+1)&
3&10&21&36&55&78&105&136&171&210\\
\end{array}
\]
}

\noindent
The row above the horizontal line contains the prescribed odd degree of exactness $2d-1$.
Rows C and D contain the minimal number of nodes we have found in the literature for the Lebesgue measure
on a square and a disk in $\R^2$. Here a ``$*$'' indicates that this is proven to be
the minimal number of nodes a quadrature rule can have and an
``O'' means that we found a quadrature rule with this number of nodes in the
literature but we did not find one with all points inside the support. The literature we used for this
is \cite{cr,co2,ck,os} and the references therein. Since the square and the disk are centrally symmetric, the Möller lower bound
from \cite{mol} applies to them. It is listed in Row B. Recently, di Dio and Schmüdgen provided lower bounds which we list in Row E.
Our lower bound from Remark
\ref{lowb}, which is much better except
for small values of $d$, is listed in Row F. Our upper bound from Corollary \ref{planecor} is given in Row G. At the moment, it seems conceivable that this upper bound could be improved to $d^2$ which
we have therefore listed in Row H. In fact, it seems to be even an open problem if moreover
the inequality from \ref{petro} can be improved to $\rh(2,2d)\le d^2$ for all $d\in\N$.
Note also that for the Lebesgue measure on a square (as for other products of two measures on $\R$), you get the
Gauss product rule with $d^2$ many nodes.


\section{New facts about Gaussian quadrature on the line}\label{line}

\noindent
Throughout the section, fix
\begin{itemize}
\item $d,m\in\N$,
\item $L\in M_{2d-1}(\R)$,
\item $\xi\in\R$,
\item a locally Lipschitz continuous function $f\colon\R^m\to\R$ such that \[\{x'\in\R^m\mid f(x')\le f(x)\}\] is compact for all $x\in\R^m$
and the generalized gradient $\partial f(x)\subseteq\R^m$
of $f$ at $x$ recalled below satisfies the following properties for every $x\in\R^n$:
\begin{itemize}
\item if $0\in\partial f(x)$, then $x=(\xi,\dots,\xi)$,
\item $\pi_i(\partial f(x))\subseteq\R_{\le0}$ if $x_i<\xi$ and
\item $\pi_i(\partial f(x))\subseteq\R_{\ge0}$ if $x_i>\xi$,
\end{itemize}
\end{itemize}
where $\pi_i:\R^m\to\R$ is the projection on the $i$-th coordinate.
It follows from Rademacher's theorem \cite[3.1.6]{fed} that there is a Lebesgue zero set $A\subseteq\R^m$
such that $f$ is differentiable at all points of $\R^m\setminus A$ and by \cite[Theorem 2.5.1]{cla}, the generalized gradient
$\partial f(x)$ at $x$ equals, independently of the choice of such a set $A$, the convex hull in $\R^m$
of all $\lim_{i\to\infty}\nabla f(x_i)$
such that $(x_i)_{i\in\N}$ is a sequence in $\R^n\setminus A$ with the properties that
$\lim_{i\to\infty}x_i=x$ and that the sequence of gradients
$(\nabla f(x_i))_{i\in\N}$ converges.
In particular, if $U\subset\R^m$ is open such that $f|_U$ is continuously differentiable on $U$, then $\partial f(x)=\{\nabla f(x)\}$
is a singleton for each $x\in U$. In general,
$\partial f(x)$ is for each $x\in\R^m$ a non-empty compact convex subset of $\R^n$ by \cite[Proposition 2.1.2]{cla}.
One might therefore also just think of a continuously differentiable function $f$ and think if its usual gradient.
This includes but also excludes important cases.
For readers that are not familiar with nonsmooth analysis it is instructive to consider \cite[Example 2.5.2]{cla}.

\bigskip\noindent
Typically, $L$ would be defined by $L:=L_{\mu,2d-1}$ for a measure $\mu$ on $\R$ with finite moments up to degree $2d-1$,
$\xi$ would be a point on the real line such that one desires to have the nodes $x_i$ of a quadrature rule for $\mu$ with a
prescribed maximum number of nodes $m$ near to $\xi$ (e.g., because $\mu$ has a lot of mass lies near $\xi$ or because $\xi$ is the mean, a median or a mode of $\mu$). The idea of the
\emph{penalty function} $f$ is that it penalizes for nodes $x_i$ far from $\xi$. The most typical
choices of $f$ satisfying the above axioms are given by
\[f(x)=|x_1-\xi|+\ldots+|x_m-\xi|\quad\text{for $x\in\R^m$}\]
where
$\partial f(x)=\conv\{z\in\{-1,1\}^m\mid z_i=\sgn(x_i-\xi)\text{ if $x_i\ne\xi$}\}$ for $x\in\R^m$,
by
\[f(x)=\max\{|x_1-\xi|,\dots,|x_m-\xi|\}\quad\text{for $x\in\R^m$}\]
where
$\partial f(\xi)=\conv\{e_1,-e_1,e_2,-e_2,\dots,e_m,-e_m\}$ and\\
$\partial f(x)=\conv\{\sgn(x_i-\xi)e_i\mid|x_i-\xi|=f(x)\}$ for $x\in\R^m\setminus\{(\xi,\dots,\xi)\}$\\
\cite[Proposition 2.3.12(b), Proposition 2.3.6(b)]{cla}, or by
\[f(x)=|x_1-\xi|^\al+\ldots+|x_m-\xi|^\al\quad\text{for $x\in\R^m$}\]
where $\al\in\R_{>1}$ is fixed and $\partial f(x)$ is the singleton
\[\partial f(x)=\left\{\al\left(\begin{smallmatrix}(\sgn(x_1-\xi))|x_1-\xi|^{\al-1}\\\vdots\\(\sgn(x_m-\xi))|x_m-\xi|^{\al-1}\end{smallmatrix}\right)\right\}\]
for $x\in\R^m$.

\bigskip\noindent
Now consider the following optimization problem:
\[\label{QOP}
\begin{array}{llll}
(Q_{L,f,m})&\text{minimize}&f(x)\\
&\text{subject to}&a_1,\dots,a_m,x_1,\dots,x_m\in\R\\
&&
\begin{array}{rcl}
\sum\limits_{i=1}^ma_i^2x_i^0&=&L(X^0)\\
&\vdots\\
\sum\limits_{i=1}^ma_i^2x_i^{2d-1}&=&L(X^{2d-1})
\end{array}
\end{array}
\]

\begin{prop}\label{optex}
Suppose that $L$ possesses a quadrature rule with not more than $m$ nodes. Then $(Q_{L,f,m})$ has an optimal solution.
\end{prop}

\begin{proof}
The feasible region of $(Q_{L,f,m})$ is non-empty by hypothesis: Indeed, choose a quadrature rule $w\colon\{x_1,\dots,x_\ell\}\to\R_{>0}$
for $L$ with $\#\{x_1,\dots,x_\ell\}=\ell\le m$. Set $a_i:=\sqrt{w(x_i)}$ for $i\in\{1,\dots,\ell\}$, $a_i:=0$ for $i\in\{\ell+1,\dots,m\}$ and choose $x_{\ell+1},\dots,x_m\in\R$ arbitrarily. Then $(a,x)\in\R^m\times\R^m$ is feasible for $(Q_{L,f,m})$. From our general assumption that
$\{x'\in\R^m\mid f(x')\le f(x)\}$ is compact, we see immediately that $(Q_{L,f,m})$ has an optimal solution because the objective
function of $(Q_{L,f,m})$ is continuous and the feasible region of $(Q_{L,f,m})$ {is} closed. 
\end{proof}

\begin{lem}\label{lemma1}
Let $(a,x)$ be an optimal solution for $(Q_{L,f,m})$ and $x\ne(\xi,\dots,\xi)$.
Then there is $h\in\R[X]_{2d-1}\setminus\{0\}$
such that for $i\in\{1,\dots,m\}$ with $a_i\ne0$
\begin{itemize}
\item $h(x_i)=0$,
\item $h'(x_i)\le0$ for all $i$ with $x_i<\xi$ and
\item $h'(x_i)\ge0$ for all $i$ with $x_i>\xi$.
\end{itemize}
\end{lem}

\begin{proof}
By a nonsmooth version of the Lagrange multiplier method \cite[Theorem 6.1.1, Remark 6.1.2(iv)]{cla}, there exist multipliers
\[\la,\la_0,\dots,\la_{2d-1}\in\R\text{, not all zero,}\]
and $z\in\R^n$ with $z\in\partial f(x)$ such that
\[
\la
\begin{pmatrix}
0\\
\vdots\\
0\\
z_1\\
\vdots\\
z_n
\end{pmatrix}
+
\la_0
\begin{pmatrix}
2a_1x_1^0\\
\vdots\\
2a_nx_n^0\\
0\\
\vdots\\
0
\end{pmatrix}
+
\la_1
\begin{pmatrix}
2a_1x_1^1\\
\vdots\\
2a_nx_n^1\\
a_1^2\\
\vdots\\
a_n^2
\end{pmatrix}
+
\ldots
+
\la_{2d-1}
\begin{pmatrix}
2a_1x_1^{2d-1}\\
\vdots\\
2a_nx_n^{2d-1}\\
a_1^2(2d-1)x_1^{2d-2}\\
\vdots\\
a_n^2(2d-1)x_n^{2d-2}
\end{pmatrix}
=0.
\]
Without loss of generality $\la\le0$ (otherwise flip the signs).
Set $h:=\sum_{i=0}^{2d-1}\la_iX^i$. Then $h\ne0$ since otherwise $\la\ne0$ and therefore $0=z\in\partial f(x)$ which contradicts
$x\ne(\xi,\dots,\xi)$ by our general assumptions on $f$. The rest now follows from $\la\le0$ and our general assumptions on $\pi(\partial f(x))$.
\end{proof}

\begin{lem}\label{lemma2}
Suppose $h\in\R[X]_{2d-1}\setminus\{0\}$, $\ell\in\N_0$ and $x_1,\dots,x_\ell\in\R$ are pairwise distinct roots of $h$ such that
\begin{itemize}
\item $h'(x_i)\le0$ for all $i$ with $x_i<\xi$ and
\item $h'(x_i)\ge0$ for all $i$ with $x_i>\xi$.
\end{itemize}
Then $\ell\le d$.
\end{lem}

\begin{proof}
We proceed by induction on $d$. If $d=1$, then $h$ can have at most one root and therefore $\ell\le1=d$. Now let $d\ge2$
and suppose the statement is already proven for $d-1$ instead of $d$. Without loss of generality suppose that $x_1<\ldots<x_\ell$ and
$\ell\ge3$. Using the properties of $h$, it is now easy to see that at least one of the following cases applies:

\medskip\noindent
\emph{Case 1:} One of the $x_i$ is a double root of $h$. Then the $x_j$ with $j\ne i$ are roots of
\[\frac h{(X-x_i)^2}\in\R[X]_{2d-3}\setminus\{0\}\]
whose derivative
\[\frac{(X-x_i)^2h'-2h(X-x_i)}{(X-x_i)^4}\in\R[X]_{2d-4}\]
evaluated at $x_j$ with $j\ne i$ has the same sign as $h'(x_i)$ because $(x_j-x_i)^2>0$.
By induction hypothesis $\ell-1\le d-1$ whence $\ell\le d$.

\medskip\noindent
\emph{Case 2:} $h'(x_1)<0$ and $h'(x_2)<0$. By the intermediate value theorem there must exist another root $x\in\R$ of $h$ with
$x_1<x<x_2$. Now the $x_j$ with $j\ge2$ are roots of
\[\frac h{(X-x_1)(X-x)}\in\R[X]_{2d-3}\setminus\{0\}\]
whose derivative
\[\frac{(X-x_1)(X-x)h'-h(2X-x_1-x)}{(X-x_1)^2(X-x)^2}\in\R[X]_{2d-4}\]
evaluated at $x_j$ with $j\ge2$ has the same sign as $h'(x_j)$ because $x_j-x_1>0$ and $x_j-x>0$.
By induction hypothesis $\ell-1\le d-1$ whence $\ell\le d$.

\medskip\noindent
\emph{Case 3:} $h'(x_{\ell-1})>0$ and $h'(x_\ell)>0$. This case is analogous to the previous one.
\end{proof}

\begin{lem}\label{unilem}
Let $w\colon N\to\R_{>0}$ be a quadrature rule for $L$ and consider
\[\La\colon\R[X]_{2d}\to\R,\ p\mapsto\sum_{x\in N}w(x)p(x).\]
Then the following are equivalent:
\begin{enumerate}[(a)]
\item $\La$ is an optimal solution of $(P_{L,X^{2d},\R})$
\item There exists $h\in P_{2d}(\R)\setminus\{0\}$ such that $\La(h)=0$.
\item $\#N\le d$
\end{enumerate}
\end{lem}

\begin{proof}
\underline{(a)$\implies$(b)}\quad Suppose (a) holds. If there exists $p\in P_{2d-2}(\R)\setminus\{0\}$ with
$L(p)=0$ then $\La(p)=L(p)=0$ and we are done. Otherwise Theorem \ref{ct} guarantees the existence
of an $h\in P_{2d}(\R)$ with $h_{2d}=X^{2d}$ and $\La(h)=0$.

\smallskip\underline{(b)$\implies$(a)}\quad Suppose $h\in P_{2d}(\R)\setminus\{0\}$ satisfies $\La(h)=0$.
Then $h$ is of even degree of at most $2d$ and has a positive leading coefficient. By multiplying
$h$ with $aX^{2k}$ for a suitable $a\in\R_{>0}$ and $k\in\N_0$, we may assume that
$h=X^{2d}-q$ for some $q\in\R[X]_{2d-1}$. Now $\La$ is feasible for $(P_{L,X^{2d},\R})$ and
$q$ for its dual $(D_{L,X^{2d},\R})$. The objective values $\La(X^{2d})$ and $L(q)$ of these feasible
solutions are equal since $\La(X^{2d})=\La(q)=L(q)$.
By conic weak duality, (a) follows.

\smallskip\underline{(b)$\implies$(c)}\quad Let $h\in P_{2d}(\R)\setminus\{0\}$ with $\La(h)=0$.
Obviously, $N\subseteq Z(h)$. But we have $\#Z(h)\le d$ since each real root of $h$ is double by
Remark \ref{nonn}.

\smallskip\underline{(c)$\implies$(b)} is clear by considering $h:=\prod_{x\in N}(X-x)^2$.
\end{proof}

\begin{lem}\label{uniuni}
$(P_{L,X^{2d},\R})$ possesses exactly one optimal solution.
\end{lem}

\begin{proof}
The existence is immediate from Proposition \ref{primsol}(a). The uniqueness follows from
the fact that we are in the one variable case: Any feasible solution $\La$ of $(P_{L,X^{2d},\R})$ is
of course determined by its objective value $\La(X^{2d})$ and by $\La|_{\R[X]_{2d-1}}=L$.
\end{proof}

\bigskip\noindent
Before we proceed, we have to catch up on introducing Gaussian quadrature.
The existence part of the following theorem is folklore and has been reproven in Examples \ref{uniex} and \ref{gaussex} above.
The uniqueness part should be well-known although we could not find a reference for the statement in its full generality in the literature.
Indeed, typical unnecessary assumptions on a measure $\mu$ on $\R$ with $L=L_{\mu,2d-1}$ often made in the literature are
that the support of $\mu$ is compact, that all moments of $\mu$ exist and are finite and that $\mu$ possesses a density with respect to
the Lebesgue measure. We therefore give a new self-contained proof for this result.

\begin{thmdf}\label{eindeutig}
There is exactly one quadrature rule $w\colon N\to\R_{>0}$ for $L$ with $\#N\le d$.
We call this quadrature rule the \emph{Gaussian quadrature rule} for $L$.
\end{thmdf}

\begin{proof}
Denote the unique optimal solution of $(P_{L,X^{2d},\R})$ that exists due to Lemma \ref{uniuni}
by $\La$. By Lemma \ref{unilem}, we can choose $h\in P_{2d}(\R)\setminus\{0\}$ with
$\La(h)=0$. We have again $\#Z(h)\le d$ since each real root of $h$ is double by
Remark \ref{nonn}. Now it suffices to show that each quadrature rule for $L$ with at most $d$ nodes
has all its nodes contained in $Z(h)$ since the point evaluations $\ev_x\colon\R[X]_d\to\R,\ x\mapsto p(x)$
($x\in Z(h)$) are linearly independent (e.g., by an argument similar to the one in the proof of Lemma \ref{lemlu}). But if $w$ is such a quadrature rule for $L$, then
$w$ is also one for $\La$ by Lemmata \ref{unilem} and \ref{uniuni}.
\end{proof}

\noindent
Now we show that the optimization problem $(Q_{L,f,m})$ leads to the just introduced Gaussian quadrature rule for $L$.

\begin{thm}\label{thmgauss}
Let $(a,x)$ be an optimal solution for $(Q_{L,f,m})$ and set
\[N:=\{x_i\mid i\in\{1,\dots,m\},a_i\ne0\}.\]
Then
\[w\colon N\to\R_{>0},\ x\mapsto\sum_{\substack{i=1\\x_i=x}}^ma_i^2\]
is the Gaussian quadrature rule for $L$.
\end{thm}

\begin{proof}
We observe that $w$ is a quadrature rule for $L$ since
\[\sum_{x\in N}w(x)x^i=\sum_{x\in N}\sum_{\substack{i=1\\x_i=x}}^ma_i^2x^i=\sum_{i=1}^na_i^2x^i=L(X^i)\]
for all $i\in\{0,\dots,2d-1\}$. According to Definition \ref{eindeutig}, it remains to show that $\#N\le d$.
But this follows from Lemmata \ref{lemma1} and \ref{lemma2}.
\end{proof}

\noindent
The main result of this section is now the following characterization of the Gaussian quadrature rule for $L$.
In disguised and slightly less general form, characterizations (b) and (d) appear already in the literature:
Characterization (b) appears in \cite[Theorem 2.2.3(1)]{ds} and much more explicitly in \cite[Theorem 1]{rb}.
Characterization (d) appears with a completely different proof in the first paragraph of the note on Page 497 in \cite{dst}.
To the best of our knowledge, characterization (c) seems to be previously unknown despite the fact that Gaussian quadrature rules have been
investigated for over two hundred years now.
Note also that one could easily add a lot of other equivalent new characterizations by choosing various other
penalty functions $f$.

\begin{thm}\label{finthm}
Let $w\colon N\to\R_{>0}$ be a quadrature rule for $L$ and $\al\in\R_{\ge1}$. Then the following are
equivalent:
\begin{enumerate}[(a)]
\item $w$ is the (unique) Gaussian quadrature rule for $L$.
\item $w$ minimizes $\sum_{x\in N}w(x)(x-\xi)^{2d}$ amongst all quadrature rules for $L$.
\item $N$ minimizes $\sum_{x\in N}|x-\xi|^\al$ amongst all {node sets of} quadrature rules for $L$.
\item $N$ minimizes $\max_{x\in N}|x-\xi|$ amongst all {node sets of} quadrature rules for $L$.
\end{enumerate}
\end{thm}

\begin{proof} Treat first the easy case where there exists $p\in P_{2d-2}(\R)\setminus\{0\}$ such that $L(p)=0$.
Then by Remark \ref{nonn}, each root of $p$ is a double root. Therefore $\#Z(p)\le d-1$. Since each quadrature rule for $L$ has
all its nodes contained in $\#Z(p)$, Theorem and Definition \ref{eindeutig} imply that there is only one quadrature rule for $L$ namely
the Gaussian one. The conditions (a)--(d) are therefore trivially all satisfied.

From now on we therefore assume that there is no $p\in P_{2d-2}(\R)\setminus\{0\}$ such that $L(p)=0$. Moreover, we replace
condition (b) by the condition
\begin{enumerate}[(a)]
\item[(b')] $w$ minimizes $\sum_{x\in N}w(x)x^{2d}$ amongst all quadrature rules for $L$
\end{enumerate}
which is equivalent since
\[\sum_{x\in N}w(x)((x-\xi)^{2d}-x^{2d})=\sum_{i=0}^{2d-1}\binom{2d}i(-\xi)^{2d-i}L(X^i)\]
depends only on $L$.

By Theorem and Definition \ref{eindeutig}, it suffices now to show that there exists a quadrature rule $w\colon N\to\R_{>0}$
minimizing the respective penalty function from (b'), (c) and (d) and that for each such one has $\#N\le d$.
For (b') this follows from Theorem \ref{genthm} together with Example \ref{rho1} (cf. also Examples \ref{uniex} and \ref{gaussex}).
For (c) and (d) this follows easily from Proposition \ref{optex} and Theorem \ref{thmgauss} choosing for $m$ the number of nodes
of any quadrature rule for $L$ and and making the choices of $f$ proposed at the beginning of the section.
\end{proof}

\begin{rem}\label{flattened}
Note that this section provides also an interesting new elementary proof for the existence of the Gaussian quadrature rule for $L$ starting from
the existence of an arbitrary quadrature rule with $m$ nodes: Take as penalty function for example
$f\colon\R^m\to\R,\ x\mapsto\sum_{i=1}^mx_i^2$. Then in Lemma 1 the usual Lagrange multiplier rule can be applied instead of
nonsmooth analysis. Moreover, since $f$ satisfies (in the ``nonsmooth notation'') even $\pi(\partial f(x))\subseteq\R_{>0}$
(instead of just $\pi(\partial f(x))\subseteq\R_{\ge0}$) for all $x\in\R^n$ with $x_i>0$, and the corresponding
symmetric condition, also Case 1 in the proof of Lemma \ref{lemma2}
is no longer needed. Indeed, the proof of Lemma 2 becomes then clear already from drawing a picture. We leave it as an
exercise to the reader to formulate the proof of Theorem \ref{thmgauss} in a short and elementary way accessible to undergraduate students
for this choice of $f$.
\end{rem}


\section{Relation to previous works}\label{previous}

\noindent
While we were working on this article, we realized that the basic idea behind the optimization problem $(P_{L,f,S})$ from Page \pageref{POP} is already present in the literature, namely in the book \cite{ds} of Dette and Studden as well as in the recent work
\cite{rb} of Ryu and Boyd. In this section, we would like to clarify the relation of our work to the work of these authors.

\bigskip\noindent
In the theory of canonical moments for measures on the real line initiated by Skibinsky \cite{ski} and developed and exposed by Dette
und Studden in \cite{ds}, optimization problems very similar to $(P_{L,X^{2d},\R})$ and variations of it with
\begin{itemize}
\item odd instead of even degree objective moment and
\item maximization instead of minimization
\end{itemize}
play a central role, see \cite[Preface]{ds}. The equivalence of (a) and (b) in
Theorem~\ref{finthm} can be easily deduced from \cite[Theorem 2.2.3(1)]{ds},
at least if $L$ comes from a measure whose support is
contained in a compact interval. In sharp contrast to our work, the techniques of Dette and Studden are, however,
limited to the one-dimensional case and use ideas from the classical theory of orthogonal polynomials.

\bigskip\noindent
Ryu and Boyd consider optimization problems very similar to $(P_{L,f,S})$. Whereas we minimize over the \emph{finite-dimensional}
cone $P_{2d}(S)^*$, they minimize, however, over the \emph{infinite-dimensional} cone of measures whose support lies in $S$.

\bigskip\noindent
Obviously, our framework would be closer to the setup of Ryu and Boyd if we had chosen to minimize over the cone $M_{2d}(S)$
which is the appropriate finite-dimensional projection of the cone Ryu and Boyd work with. It is easy to see that this would also have
been a possible way to go since $P_{2d}(S)^*$ is the closure of $M_{2d}(S)$ by Proposition \ref{momclo} and under very mild
assumptions each optimal solution of $(P_{L,f,S})$ lies in $M_{2d}(S)$ anyway by Proposition \ref{primsol}.
That we chose to formulate $(P_{L,f,S})$ with the cone $P_{2d}(S)^*$ instead of
$M_{2d}(S)$ has at least two advantages: First, we can then directly apply the duality theory of conic programming
(for closed cones which happily are not required to have interior in Renegar's book \cite{ren}) to $(P_{L,f,S})$. Second,
one can directly implement $(P_{L,f,S})$ in certain (but rather rare) cases as a semidefinite program \cite[4.6.2]{bv}
using so-called moment matrices
\cite[Subsection 4.1]{la1,la2}. These cases
include
\begin{itemize}
\item $S=\R$,
\item $d=1$, and
\item $S=\R^2$ and $d=2$
\end{itemize}
since in each of these cases $P_{2d}(S)$ consists of sums of squares of polynomials by Hilbert's 1888 theorem (see \cite{hil} or
\cite[Theorem 3.4]{la1,la2}) the most trivial part of which we have proven in Proposition \ref{sos2}. A related but different approach to computing quadrature rules with
semidefinite programming has recently been found in \cite{abm}.

\bigskip\noindent
The first real difference between our work and the work of Ryu and Boyd is, however, that we minimize over a \emph{finite-dimensional cone}
(which is $P_{2d}(S)^*$ but could also have been the finite-dimensional projection $M_{2d}(S)$ of their cone). Whereas we extend
$M_{2d}(S)$ to $P_{2d}(S)^*$ in order to be able to apply the duality theory of conic programming, they have to
extend the cone of Borel measures to the cone of \emph{Borel charges} \cite[Section 10]{ab} (which are finitely additive but not necessarily countably
additive) in order to use infinite-dimensional
Lagrange duality. {These}
arguments require rather deep facts from infinite-dimensional analysis like \cite[Theorems 4 and 5]{rb} whereas
we use elementary facts from finite-dimensional convex analysis instead.

\bigskip\noindent
The main difference between our work and \cite{rb} is again that their \emph{theoretical} results are
essentially again restricted to the equivalence of (a) and (b) in Theorem \ref{finthm} and in particular do not concern the multivariate
case. Ryu and Boyd reveal, however, the
advantages and the practical relevance of the optimization theoretic viewpoint and they conduct numerical experiments in
the multi-dimensional case by finding linear programs that approximate their optimization problem. Translated into
our framework, what they do is that they try to exchange $P_{2d}(S)^*$ by $P_{2d}(G)^*$ where $G$ is a fine grid inside $S$
(which is in practice only possible for low-dimensional compact $S$).
Since $G$ is finite, $P_{2d}(G)^*$ is obviously a polyhedron and $(P_{L,f,G})$ is a linear program. It turns out that this yields in
practice quadrature rules whose nodes are still numerous but usually seem to form only a few clusters inside the grid $G$.
Ryu and Boyd explain how these clusters can in practice
often be merged to a single node no longer necessarily
contained in the grid $G$ without changing the prescribed moments using sequential quadratic programming. However, they do not give a
theoretical upper bound on the number of nodes obtained. Indeed, in the example given in \cite[Subsection 5.2, Figure 2]{rb} they obtain
a quadrature rule of degree $5$ (i.e., $d=3$ so that $2d-1=5$) with $12$ nodes for a measure on $\R^2$ (i.e., $n=2$)  although we
know by our Theorem \ref{planethm} that $\frac32d(d-1)+1=10$ nodes would suffice.
Since in this example there are three pairs of very close
nodes, one could, however, speculate if $9$ nodes would perhaps be sufficient but were not found by the heuristic of Ryu and Boyd.

\bigskip\noindent
Ryu and Boyd begin to discuss how the method relates to the idea of obtaining sparse solutions by $\ell_1$-norm minimization
underlying methods like compressed sensing. This interesting question should now also be viewed in the light of Theorem \ref{finthm}.

\section*{Note added in proof}
\noindent
{
While this paper was in review, di Dio and Schmüdgen released their related preprint \cite{dds}. In \cite[Theorem 57]{dds},
they show in our language that each linear form in
$P_{2d}(\R^2)^*\supseteq M_{2d}(\R^2)$ possesses a \emph{generalized} quadrature rule with at most
$\frac32d(d+1)+1$ many nodes. It is interesting to compare this to our Corollary~\ref{planecor}. Whereas we treated the odd degree
case and we put a lot of energy in avoiding nodes at infinity, di Dio and Schmüdgen treat the even degree case and allow nodes at
infinity. In the odd degree case discussed in this work, they proved however new interesting lower bounds \cite[Theorem 27 and 53]{dds}.
We included this in our discussion in Section \ref{plane} shortly before this article went into print.
}

\section*{Acknowledgments}
\noindent
Important ideas towards this work have been worked out at the authors' stay at the
Issac Newton Institute during the Programme on Polynomial Optimisation in 2013,
the second author's visit at the Aalto Science Institute in 2014 and
the first author's postdoctoral fellowship at the Fields Institute during the
Thematic Program on Computer Algebra in 2015.
The authors thank all these institutions for their generous hospitality and support.
Finally, the authors would like to thank Greg Blekherman,
Erwan Brugallé, Jean Lasserre, Antonio Lerario, Mihai Putinar, {Johannes Rau}, Bruce Reznick, Konrad Schmüdgen and Ronan Quarez
for interesting discussions on this topic.

\end{document}